\documentclass[11pt,a4paper]{amsart}
\usepackage{amsmath,amsfonts,amsthm,amsopn,color,amssymb,enumitem}
\usepackage{palatino}
\usepackage{graphicx}
\usepackage{caption}
\usepackage{subcaption}
\usepackage[colorlinks=true]{hyperref}
\hypersetup{urlcolor=blue, citecolor=red, linkcolor=blue}

\usepackage{esint}
\usepackage[title]{appendix}

\newtheorem{theorem}{Theorem}[section]
\newtheorem{lemma}[theorem]{Lemma}

\newtheorem{proposition}[theorem]{Proposition}
\newtheorem{remark}[theorem]{Remark}
\newtheorem{remarks}[theorem]{Remarks}

\newcommand{\Sd}{\mathbb{S}^2}
\newcommand{\Hd}{\mathbb{H}^2}

\newcommand{\e}{\varepsilon}

\newcommand{\R}{\mathbb{R}}

\newcommand{\D}{\mathbb{D}}
\newcommand{\oD}{\overline{\mathbb{D}}}
\newcommand{\de}{\partial}
\newcommand{\weakto}{\rightharpoonup}

\DeclareMathOperator{\dist}{dist}

\renewcommand{\b }{\beta }

\newcommand{\G}{\Gamma}

\renewcommand{\S}{\Sigma}

\def\bbm[#1]{\mbox{\boldmath $#1$}}
\newcommand{\beq }{\begin{equation}}
\newcommand{\eeq }{\end{equation}}

\def\sideremark#1{\ifvmode\leavevmode\fi\vadjust{\vbox to0pt{\vss
 \hbox to 0pt{\hskip\hsize\hskip1em
 \vbox{\hsize3cm\tiny\raggedright\pretolerance10000
  \noindent #1\hfill}\hss}\vbox to8pt{\vfil}\vss}}}%

\begin{document}

\title[A curvature prescription problem on the disk]{ Blow-up analysis of conformal metrics of the disk with prescribed Gaussian and geodesic curvatures}

\author{Aleks Jevnikar}
  \address{Aleks Jevnikar \\ 
University of Udine\\
Via delle Scienze 206\\
33100 Udine, Italy.}
  \email{aleks.jevnikar@uniud.it}

\author{Rafael L\'{o}pez-Soriano}
  \address{Rafael L\'{o}pez-Soriano \\
    Universidad de Valencia\\
    Departamento de An\'alisis Matem\'atico\\
    Dr. Moliner 50 \\
    46100 Burjassot (Valencia), Spain.}
  \email{rafael.lopez-soriano@uv.es}

\author{Mar\'ia Medina}
  \address{Maria Medina \\
    Universidad de Granada\\
    Departamento de An\'alisis Matem\'atico\\
    Campus Fuentenueva\\
    18071 Granada, Spain}
  \email{mamedina@ugr.es}

\author{David Ruiz}
  \address{David Ruiz \\
    Universidad de Granada\\
    Departamento de An\'alisis Matem\'atico\\
    Campus Fuentenueva\\
    18071 Granada, Spain}
  \email{daruiz@ugr.es}

\thanks{R. L.-S. and D. R. have been supported by the FEDER-MINECO Grants MTM2015-68210-P  and PGC2018-096422-B-I00 and by J. Andalucia (FQM-116).  R.L.-S. is currently funded under a Juan de la Cierva Formaci\'on fellowship (FJCI-2017-33758) by the Ministry of Science, Innovation and Universities. This paper is also part of the project ``The prescribed Gaussian and geodesic curvatures problem'' funded by Mathematisches Forschungsinstitut Oberwolfach. M. M. is supported by the European Union's Horizon $2020$ research and innovation programme under the Marie Sklodowska-Curie grant agreement N $754446$ and UGR Research and Knowledge Transfer Fund - Athenea3i.}




\begin{abstract}
{This paper is concerned with the compactness of metrics of the disk with prescribed Gaussian and geodesic curvatures. We consider a blowing-up sequence of metrics and give a precise description of its asymptotic behavior. In particular, the metrics blow-up at a unique point on the boundary and we are able to give necessary conditions on its location. It turns out that such conditions depend locally on the Gaussian curvatures but they depend on the geodesic curvatures in a nonlocal way. This is a novelty with respect to the classical Nirenberg problem where the blow-up conditions are local, and this new aspect is driven by the boundary condition.}
\end{abstract}

\maketitle

\section{Introduction}
\setcounter{equation}{0}

Let $(\S,g)$ denote a compact surface $\S$ equipped with a certain metric $g$. The classical Kazdan-Warner problem (see \cite{ber, KW})  consists of determining whether a prescribed function $K$ is the Gaussian curvature of a new metric $\tilde{g}$ conformal to $g$. If $\tilde{g}=e^u g$ and $K$, $K_g$ are respectively the Gaussian curvatures relative to these metrics, then the following relation holds:
\begin{equation}\label{eq:regu}
-\Delta_g u + 2K_g(x) = 2K(x)e^u \quad\mbox{in}\quad \S.
\end{equation}
Hence the Kazdan-Warner problem reduces to solving this equation.

%


\medskip 

The problem of prescribing the Gaussian curvature on the standard sphere $\mathbb{S}^2$ is particularly delicate and receives the name of Nirenberg problem. This question has been addressed by a large number of papers (see for instance \cite{aubin, ChGYg, ChYgActa, ChYg, chen-ding, chen-ding2, C-L CPAM,  Ji, han, han-li, hamza,  struwe, xu-yang}), some of whom are commented below. For instance, an obstruction to existence was found by Kazdan and Warner (\cite{KW}). On the other hand, if $K$ has antipodal symmetry (i.e. $K(-x) = K(x)$) and it is somewhere positive then there exists a solution, see \cite{Moser}. Other results under symmetry assumptions are given in \cite{chen-ding}. In \cite{struwe}  solutions of \eqref{eq:regu} have been found as the asymptotic limit of a suitable flow.

%

A related and significant issue in this kind of problems is the study of compactness of solutions, starting from \cite{breme, li-sha}. Roughly speaking, given $\{u_n\}$ a sequence of solutions, we desire to find conditions that allow us to pass to the limit. By regularity, it is enough to show $L^{\infty}$ boundedness. This problem is typically studied by means of a blow-up analysis, which determines whether the sequence $u_n$ is uniformly bounded or may blow-up. After rescaling, a blowing-up sequence of solutions resembles locally a limit entire solution, and such solutions are classified in \cite{C-L}. In particular this implies a quantization result for general Liouville equations, see \cite{li-sha}.

Being more specific, let us consider the problem:
\begin{equation}\label{eq:regun}
-\Delta u_n + 2 = 2K_n(x)e^{u_n} \quad\mbox{in}\quad \mathbb{S}^2,
\end{equation}
where $K_n$ converges in $C^2$ sense to a strictly positive function $K(x)$. Observe that if $K_n=1$, the problem is invariant by the group of conformal maps of the sphere, which is not compact. In the non-constant case, this invariance is lost but concentration of solutions may still occur. This is the so-called ``bubbling phenomena"; the solutions concentrate all their mass around a certain singular point. 

Not only the asymptotic behavior of the sequence is relevant, but also the location of the point of concentration. It has been shown (\cite{ChYg2, ChGYg}) that if the sequence $u_n$ is blowing-up, then the point of concentration satisfies:
\begin{equation} \label{cond-nir}  \nabla K(p)=0,\ \Delta K(p)=0.  \end{equation}

As a consequence, if those conditions are never satisfied for any point $p \in \mathbb{S}^2$, one concludes compactness for \eqref{eq:regun}. This result is also a key point in the proof of existence of solutions for the Nirenberg problem, which was given for the first time in \cite{ChYgActa} via a variational argument and revisited under a different approach in \cite{Ji}.

\medskip

If $\S$ is a surface with boundary, one usually imposes boundary conditions. A natural geometric problem consists in prescribing also the geodesic curvature of the boundary; in this way we are led to the problem:
\begin{equation} \label{gg0}
\left\{\begin{array}{ll}
-\Delta u +2 K_g = 2 K e^u  & \text{ in } \S, \\
\frac{\partial u}{\partial \nu} +2 h_g = 2he^{u/2}  &\text{ on } \partial\S,\end{array}\right.
\end{equation}
where $\nu$ is the outward normal vector to $\partial \S$. Here $h_g$ is the  geodesic curvature of $\partial \S$ under the metric $g$, and $h$ is the geodesic curvature to be prescribed for the new metric $\tilde{g}=e^u g$. 

\medskip

In the literature there are some results on the latter problem.  The case of constant $K$, $h$ has been considered in \cite{brendle}, where the author used a parabolic flow to obtain solutions in the limit. Some classification results for the half-plane are also available in \cite{galvez, li-zhu, Zhang}. The case of nonconstant curvatures was addressed for the first time in \cite{cherrier}, but the results there are partial and in some of them an unknown Lagrange multiplier appears in the equation.

In \cite{lsmr} the case of surfaces topologically different from the disk is studied when $K<0$. In that paper a new type of blow-up phenomenon appears, where length and area diverge. In the existence results the authors exploit the variational formulation of the problem and solutions are obtained by minimization and min-max techniques. For more general mean field problems with boundary terms a quantization result has been recently given in \cite{bls}.

\medskip

The case of the disk $\S = \D$ can be seen as a natural generalization of the Nirenberg problem. Indeed, the effect of the noncompact group of conformal maps of the disk plays a fundamental role. The problem becomes:
\begin{equation} \label{gg}
\left\{\begin{array}{ll}
\displaystyle{-\Delta u = 2Ke^u } \qquad & \text{ in }\D, \\
\displaystyle{\frac{\partial u}{\partial \nu} +2 = 2he^{u/2}} \qquad  &\text{ on } \partial \D.
\end{array}\right.
\end{equation}

Integrating \eqref{gg} and applying the Gauss-Bonnet Theorem, one obtains
\beq\label{GB}
\int_{\D}K e^u \, +  \int_{\partial \D}h e^{u/2} \, =  2\pi,
\eeq
which implies that $K$ or $h$ must be positive somewhere.  

Some works have also addressed problem \eqref{gg}. For example, the case $h=0$ has been treated in \cite{ChYg} (see also \cite{guo-hu}), while the case $K=0$ has been considered in \cite{kcchang, LiLiu, LiHu}. If $K=0$ a blow-up analysis has been performed in \cite{GuoLiu}, and in \cite{DLMR} a new approach is given under mild conditions on the function $h$.

Up to our knowledge, the only result on \eqref{gg} for non-constant curvatures is \cite{CruzRuiz}. In that paper an existence result for \eqref{gg} with positive symmetric curvatures is given, in the spirit of the aforementioned result of Moser (\cite{Moser}).

\medskip

The aim of this paper is to provide a complete blow-up analysis of problem \eqref{gg} for non-constant functions $K$, $h$. We emphasize that we do not impose any sign restriction on the functions $K$ and $h$. We are interested in the precise asymptotic behavior of blowing-up solutions, but also on the location of the point of concentration, which turns out to be located on $\partial \D$. Indeed, we shall find an analogue of the conditions in \eqref{cond-nir} for equation \eqref{gg}. Such conditions involve both curvatures $K$ and $h$, and studying their interaction is the original motivation of this work. It turns out that the blow-up point $p$ depends on $K$ in a local way but it depends on $h$ in a nonlocal fashion. Indeed the 1/2 laplacian of $h$ appears which, as is well known, depends on all values of $h$ on $\partial \D$. This  was rather unexpected, at least for us, since the original problem is local in nature. 

\medskip Being more specific, our main result is the following:


%
%
%

\begin{theorem} \label{main} Let $u_n$ be a sequence of solutions of the problem
	\begin{equation} \label{ggn}
	\left\{\begin{array}{ll}
	\displaystyle{-\Delta u_n = 2K_ne^{u_n} } \qquad & \text{in $\D$},\\
	\displaystyle{\frac{\partial u_n}{\partial \nu} +2 = 2h_ne^{u_n/2}} \qquad  &\text{on $\partial \D $},
	\end{array}\right.
	\end{equation}
where $K_n \to K$ in $C^2(\oD)$ and $h_n \to h$ in $C^2(\partial \D)$ as $n\to +\infty$. Assume moreover that $u_n$ is blowing-up, namely, $\sup \{ u_n \} \to +\infty$, with bounded global mass, 
	\begin{equation} \label{mass} \int_{\D} e^{u_n} \, +  \int_{\partial \D} e^{u_n/2} \leq C.	\end{equation}

	Then there exists a unique point $p \in \partial \D$ such that 
	
	\begin{enumerate}
		\item[i)] If $K(p) \leq 0$ then $h(p)> \sqrt{- K(p)}$.	
		
		\item[ii)] \noindent There exists $a_n \to p$ so that
		$$u_n(z)= u_{a_n}(z) + \psi_n(z),$$
		$$u_{a_n}(z):=2\log \left\{\frac{2 \hat{\phi}_n (1-|a_n|^2)}{ \hat{\phi}_n^2|1- \overline{a}_n z|^2 + \hat{k}_n |z-a_n |^2}\right\},$$
		with
		$$ \hat{\phi}_n:= \phi_n \left(\frac{a_n}{|a_n|} \right), \qquad \hat{k}_n:=K_n \left(\frac{a_n}{|a_n|} \right),$$
		where 	
		\beq\label{phin}
		\phi_n(z):= h_n(z) + \sqrt{h_n(z)^2 + K_n(z)}.
		\eeq 
		Moreover, the error term $\psi_n$ satisfies:
		$$ \| \psi_n \|_{C^{0,\alpha}(\overline{\D})} \to 0 \ \mbox{ for any } \alpha \in (0,1/2).$$
		\item[iii)] Let $H$ denote the harmonic extension of $h$, that is,		
		\begin{equation*} 
		\left\{\begin{array}{ll}
		\displaystyle{\Delta H = 0} \qquad & \text{in $\D$},\\
		H = h \qquad  &\text{on $\partial \D $.}
		\end{array}\right.
		\end{equation*}
		
		Define:	
		\begin{equation}\label{PHI}
		\Phi(z):= H(z) + \sqrt{H(z)^2 + K(z)},
		\end{equation}
		 which, by i), is positive and well defined at least in a neighborhood of $p$. Then:  
		 \beq \label{cond-nir2} \nabla \Phi(p)=0. \eeq
		
	\end{enumerate}
\end{theorem}

\medskip 

Let us observe that condition \eqref{cond-nir2} above is equivalent to:
$$  h_\tau(p) =- \frac{K_\tau(p)}{2\Phi(p)} ,\quad   (-\Delta)^{1/2} h(p) =- \frac{K_\nu(p)}{2 \Phi(p)}.$$

Here $\tau$ denotes the tangent vector to $\partial \D$ at $p$. As a consequence we obtain a compactness criterion: if such a point $p$ does not exist, then the sequence $u_n$ is bounded from above and it is precompact. As mentioned before, this criterion involves local terms on $K$, $h$, but also the half-laplacian of $h$ at the blow-up point.

Observe that if $K=0$, then $h(p)>0$ and condition $iii)$ reads as $h_{\tau}(p)=0$, $(-\Delta)^{1/2}h(p)=0$ (in other words, $\nabla H(p)=0$). Instead, if $h= 0$ then $K(p)>0$ and we obtain that $\nabla K(p)=0$. Even these particular results were not known in the literature.

\medskip 

The proof of Theorem \ref{main} involves a quite detailed blow-up analysis. Indeed one needs good global estimates on the blow-up sequence, and not only local estimates around the blow-up point. By a suitable rescaling, we can pass to a limit problem posed in a half-plane, whose solutions have been classified. But for a global estimate one needs to make use of the conformal group of the disk and make this approach match with the previous rescaling argument. 

In our proofs we have to bypass several technical difficulties since no assumption on the sign of $K$, $h$ is made; the possible compensation of terms gives some troubles when passing to the limit. In general this can be very problematic, see \cite{dmls, dmlsr}. We are able to overcome these difficulties by exploiting the finite mass assumption and using some ideas from \cite{lsmr}.

\begin{remarks}$ $
	
\begin{enumerate}
	
\item Condition \eqref{mass} has a clear geometric interpretation: both the area and the length are assumed to be bounded. This hypothesis is necessary for Theorem \ref{main} to hold. Otherwise, a different phenomenon blow-up could appear, as has been shown in \cite{lsmr} for $K<0$. Observe that if $K>0$ and $h>0$, assumption \eqref{mass} is automatically satisfied by \eqref{GB}. 

\item As commented above, condition \eqref{cond-nir2} is an analogue of \eqref{cond-nir} for problem \eqref{ggn}. It is to be expected that one can use it to give existence results for equation \eqref{gg}. This will be pursued in a future work.

\item Condition \eqref{cond-nir2} admits a variational interpretation, see Remark \ref{interpretation}.

\end{enumerate}

\end{remarks}

The rest of the paper is organized as follows. In Section \ref{sec:prelim} we collect and prove some preliminary results. In particular the limit problems in the disk, in the plane and in the half-plane are considered, together with a Pohozaev-type identity. In Section \ref{sec:blowup} we start the blow-up analysis and we prove $i)$ of the Theorem \ref{main}. This analysis is refined in Section ~\ref{error} by giving a precise description of the asymptotic behavior of $u_n$: namely, assertion $ii)$ of Theorem \ref{main}. In Section \ref{sec:proof} we give the proof of condition $iii)$ in Theorem \ref{main}. Some asymptotic computations have been postponed to a final Appendix, \ref{Appendix}.

\medskip

{\bf Notations.}

Let us fix some notations. The metric distance between two points $z_1,z_2\in\overline{\D}$ will be written as $\dist(z_1,z_2)$. We will denote an open ball centered at a point $p\in\overline{\D}$ of radius $r>0$ as 
$$B_r(p):=\{z\in\overline{\D}:\,\dist(z,p)<r\}.$$

At any point $z=(x,y) \in \partial \D$ we fix a tangent vector $\tau(x,y)=(-y, x)$ or, in complex notation, $\tau(z) = i z$.
We will use the following notation for some subsets of $B_r(p)\subset\R^2$:
\begin{eqnarray*}
B_r^+(p)&:=&\left\{z=(x,y)\in \R^2:\,\dist(z,p)<r,\,y\geq0\right\};\\
\G_r(p)&:=&\left\{(x,y)\in\de B_r^+(p):\,y=0\right\};\\
\de^+B_r(p)&:=&\de B_r^+(p)\setminus\G_r(p).
\end{eqnarray*}
The same notions will be used in $\mathbb{R}^2$ with its underlying metric.

In the estimates we will denote $C$ as a positive constant, independent of the parameters, that may vary from line to line. If some dependence respect to certain parameters must be pointed out, we will indicate it in the subscript, such as $C_\varepsilon$ or $C_{\varepsilon,\delta}$.

\

\section{Preliminaries} \label{sec:prelim}
\setcounter{equation}{0}

In this section we collect and derive some useful results which will be used later on. 

\subsection{The limit problem in the disk}

We are devoted to the properties of the limit problem:
\begin{equation} \label{limit}
\left\{\begin{array}{ll}
\displaystyle{-\Delta v = 2K_0e^v } \qquad & \text{in $\D$}, \\
\displaystyle{\frac{\partial v}{\partial \nu} +2 = 2h_0 e^{v/2}} \qquad  &\text{on $\partial \D$},
\end{array}\right.
\end{equation}
where $K_0$, $h_0$ are real constants. The content of this section is rather known, but we have not been able to find a specific reference. 

Let us introduce the group of conformal maps of the disk, namely:
\begin{equation} \label{conformal group} \mathcal{G}:= \{ e^{i \theta} f_a : \oD \to \oD;  \ \theta \in [0, 2 \pi), \ f_a(z)= \frac{a+z}{1+ \overline{a}z}, \ a \in \D\}.\end{equation}

It is well-known that problem \eqref{limit} is conformally invariant, i.e., given a solution $v$ of \eqref{limit} and $f \in \mathcal{G}$, then:
$$ v_f(z) := v(f(z)) + 2 \log |f'(z)|,$$
is also a solution of the same problem. In the next lemma we show that the unique solutions of \eqref{limit} are those coming from conformal maps from the disk to the standard surfaces with constant curvatures. 

\begin{lemma} Problem \eqref{limit} admits a solution if and only if:
	\begin{equation} \label{cond}  \mbox{ either }K_0 >0  \ \ \mbox{ or } \ \ K_0 \leq 0 \mbox{ and } h_0 > \sqrt{-K_0}. \end{equation} 
	In such case, all solutions are determined by the formula
	\begin{equation} \label{limitprofile} v_{a}(z):=2\log \left\{\frac{2 \phi_0 (1-|a|^2)}{ \phi_0^2|1 - \overline{a} z|^2 + K_0 |z-a|^2}\right\}, \end{equation}
	where $\phi_0:= h_0 + \sqrt{h_0^2 + K_0}$, and $a \in \D$. 
\end{lemma}

\begin{proof}
	If $K_0 \neq 0$, by considering the function $v + \log |K_0|$, we pass to a problem with $K_0 = \pm 1 $. Then, we can restrict ourselves to the cases $K_0 = \pm 1$ or $K_0=0$. Observe that under the metric $g= e^v dz$, $\D$ has constant gaussian curvature equal to $K_0$. Hence it is locally isometric to compact subdomains of $\Sd$, $\Hd$ or $\R^2$, depending on the case $K_0=1$, $K_0=-1$, or $K_0=0$, respectively.
	
	Observe also that $\partial \D$ has constant geodesic curvature equal to $h_0$, and this property translates via the local isometry. Since $\D$ is simply connected, we conclude that the local isometry is a global one. In ther words, we have a global isometry $\Upsilon: \D \to U$, where $U$ is a subdomain of either  $\Sd$, $\Hd$ or $\R^2$, with geodesic curvature equal to $h_0$. 
	
	Observe that if the ambient domain is $\R^2$, then $h_0$ needs to be strictly positive. Moreover, in $\Hd$ such domains are bounded only if $h_0 >1$. As a consequence we obtain \eqref{cond}.
	
	In other words, $U$ is a disk in either $\Sd$, $\Hd$ or $\R^2$ and, by composing with a symmetry, we can assume that:
	$$ U = \left \{ \begin{array}{lr} \left \{(x_1, x_2, x_3 ) \in \Sd:\ x_3 \geq \frac{h_0}{\sqrt{1+h_0^2}} \right \} & \mbox{ if } K_0=1, \\ \\  \left \{(x_1, x_2) \in \Hd:\ \sqrt{x_1^2 + x_2^2} \leq h_0 - \sqrt{h_0^2-1} \right \} & \mbox{ if } K_0=-1, \\ \\ \left \{(x_1, x_2) \in \R^2:\ \sqrt{x_1^2 + x_2^2} \leq 1/h_0 \right \} & \mbox{ if } K_0=0.  \end{array} \right.$$
Above we have expressed $\Hd$ via coordinates in the Poincar\'{e} disk.
	
	On the other hand, the identity map $I: (\D, dz) \to (\D, e^v dz)$ is clearly a conformal map. As a consequence, the composition $\Upsilon \circ I: (\D, dz) \to U$ is a conformal map. Those maps are classified, and hence the conformal factor $e^v$ is given by one of those maps. This gives the expression \eqref{limitprofile}.
\end{proof}

\begin{remark} \label{rem rad} In particular, there exists a unique solution to \eqref{limit} satisfying the extra condition:
$$ \int_{\D} x e^{v(z)} =0, \ \int_{\D} y e^{v(z)} =0,$$
which is just the solution given in \eqref{limitprofile} for $a=0$, that is:
\begin{equation} \label{radialprofile} v_0(z) = 2\log \left\{\frac{2 \phi_0}{ \phi_0^2 + K_0 | z |^2}\right\}.\end{equation}
\end{remark}

Next lemma addresses the question of nondegeneracy of such solution:

\begin{lemma} \label{lema linear 1} Let us consider the linearized problem:
	\begin{equation} \label{linear}
	\left\{\begin{array}{ll}
	\displaystyle{-\Delta \psi = 2K_0 e^{v_0} \psi } \qquad & \text{in $\D$}, \\
	\displaystyle{\frac{\partial \psi}{\partial \nu} =2 h_0 e^{v_0/2} \psi } \qquad  &\text{on $\partial \D $},
	\end{array}\right.
	\end{equation}
	where $v_0$ is given in \eqref{radialprofile}. Then the vector space of solutions of \eqref{linear} is spanned by the functions:
	$$ \psi_1(z):= \frac{x}{\phi_0^2 + K_0 |z|^2}, \ \psi_2(z):= \frac{y}{\phi_0^2+ K_0|z|^2},$$
	where $\phi_0:=h_0+\sqrt{h_0^2+K_0}$.
	In particular, a solution of \eqref{linear} satisfying the extra assumptions:
	$$ \int_{\D} x e^{v_0(z)} \psi(z) =0, \ \int_{\D} y  e^{v_0(z)} \psi(z) =0,$$
must be necessarily equal to $0$.
\end{lemma}

\begin{proof}
	
	We show the proof only in the case $K_0 =1$, the other cases being analogous. As shown in the previous lemma, $(\D, e^{v_0} dz)$ is isomorphic to a spherical cap in $\Sd$, given by $U=\{(x_1, x_2, x_3 ) \in \Sd:\ x_3 \geq \frac{h_0}{\sqrt{1+h_0^2}}\}$. Via that conformal map, problem \eqref{linear} becomes:
	\begin{equation}
	\left\{\begin{array}{ll}
	\displaystyle{-\Delta \psi = 2 \psi } \qquad & \text{in $U$}, \\
	\displaystyle{\frac{\partial \psi}{\partial \nu} =2 h_0  \psi } \qquad  &\text{on $\partial U $.}
	\end{array}\right.
	\end{equation}
	
	As it is well known, the coordinate functions $x$, $y$ are generators of the vector space of solutions of that problem. Via the conformal map, this translates into $\psi_1$, $\psi_2$.
\end{proof}

Next lemma is a standard regularity result:

\begin{lemma} \label{regularity} Let $\psi$ be a solution of
	\begin{equation} 
\left\{\begin{array}{ll}
\displaystyle{-\Delta \psi =  c(z)} \qquad & \text{in $\D$}, \\
\displaystyle{\frac{\partial \psi}{\partial \nu} =  d(z)} \qquad  &\text{on $\partial \D $.}
\end{array}\right.
\end{equation}
Then for any $q>1$ there exists $C_q>0$ such that  
$$ \|\psi\|_{W^{1+1/q, q}(\D)} \leq C_q (\| c \|_{L^q(\D)} + \| d \|_{L^q(\partial \D)} ).$$
\end{lemma}
	
The next result is a quantitive version of the nondegeneracy result stated in Lemma \ref{lema linear 1}.

\begin{lemma} \label{lema linear 2} For any $q>1$ there exists $C_q>0$ such that the following holds: for any solution $\psi$ of the nonhomogeneous linearized problem
	\begin{equation} \label{linearbis}
	\left\{\begin{array}{ll}
	\displaystyle{-\Delta \psi = 2K_0 e^{v_0} \psi + c(z)} \qquad & \text{in $\D$}, \\
	\displaystyle{\frac{\partial \psi}{\partial \nu} = 2h_0 e^{v_0/2} \psi + d(z)} \qquad  &\text{on $\partial \D $},
	\end{array}\right.
	\end{equation}
satisfying 
	\begin{equation} \label{baricentro} \int_{\D} x e^{v_0(z)} \psi(z) =0, \ \int_{\D} y  e^{v_0(z)} \psi(z) =0,\end{equation}
we have that 
	$$ \|\psi\|_{W^{1+1/q, q}(\D)} \leq C_q (\|c\|_{L^q(\D)} + \|d\|_{L^q(\partial \D)} ).$$
\end{lemma}

\begin{proof} Consider the solution $\psi_n$ to the problem 
\begin{equation*}
	\left\{\begin{array}{ll}
	\displaystyle{-\Delta \psi_n = 2K_0 e^{v_0} \psi_n + c_n(z)} \qquad & \text{in $\D$}, \\
	\displaystyle{\frac{\partial \psi_n}{\partial \nu} = 2h_0 e^{v_0/2} \psi_n + d_n(z)} \qquad  &\text{on $\partial \D $.}
	\end{array}\right.
	\end{equation*}
	We assume that the solutions $\psi_n$ are reasonably smooth, otherwise one can argue by density. Reasoning by contradiction, suppose that there exists $c_n$, $d_n$ and solutions $\psi_n$ with 
$$\|c_n\|_{L^q(\D)} + \|d_n\|_{L^q(\partial\D)}=1,\qquad \| \psi_n \|_{W^{1+1/q,q}(\D)} \to +\infty.$$
Define $\tilde{\psi}_n := \|\psi_n\|_{W^{1+1/q,q}(\D) }^{-1}\psi_n$. Up to a subsequence, we can assume that $\tilde{\psi_n} \weakto \tilde{\psi_0}$ in $W^{1+1/q,q}(\D)$. Clearly, by compactness,
	$$ 2 K_0 e^{v_0} \tilde{\psi}_n +\frac{c_n}{\|\psi_n\|_{W^{1+1/q,q}(\D)}} \to  2 K_0 e^{v_0} \tilde{\psi}_0 \ \mbox{ in } L^q(\D), $$ $$ 2h_0 e^{v_0/2} \tilde{\psi}_n +\frac{d_n}{\|\psi_n\|_{W^{1+1/q,q}(\D)}}  \to  2h_0 e^{v_0/2} \tilde{\psi}_0 \ \mbox{ in } L^q(\partial \D).$$
By Lemma \ref{regularity} we conclude $\tilde{\psi}_n \to \tilde{\psi}_0$ in  $W^{1+1/q,q}(\D)$. In particular, $\| \tilde{\psi}_0\|_{W^{1+1/q, q}(\D)}=1$. Moreover, if $\psi_n$ satisfy \eqref{baricentro}, so it does $\tilde{\psi}_0$. By Lemma \ref{lema linear 1} we conclude that $\tilde{\psi}_0=0$, a contradiction.
\end{proof}

\subsection{The limit problem in the plane and the half-plane}\label{sublimit}

As it is well known, the study of blowing-up solutions of a nolinear PDE can be put in relation with limit problems in the entire space or in half-spaces. The classification of the finite mass solutions of Liouville problems was first made in \cite{C-L} in the whole space $\R^2$, and is rather known. Indeed, the unique solutions of the problem:
\begin{equation}\label{limit2}
-\Delta v = 2 K_0 e^v \quad \text{in $\mathbb{R}^2 $} , \qquad  \int_{\R^2} e^v  < +\infty,
\end{equation}
are given by the expression:
\beq\label{entirehp2}
v(x)= 2 \log \left\{ \frac{2 \lambda}{ K_0 \lambda^2 + |x-x_0|^2} \right\},\ \ x_0 \in \R^2, \ \lambda >0.
\eeq
All those solutions are known to satisfy the quantization property:
\begin{equation} \label{quantization2} K_0 \int_{\R^2} e^v = 4 \pi. \end{equation}

Since our problem considers a boundary condition, we will also be interested in this paper in the solutions of the problem posed in the halfplane: that is, given real constants $K_0, h_0$, the problem
\begin{equation}\label{limitproblem}
\left\{\begin{array}{ll}
-\Delta v = 2 K_0 e^v \qquad & \text{in $\mathbb{R}^2_+ $,} \vspace{0.3cm} \\
\frac{\partial v}{\partial \nu} = 2 h_0 e^{\frac{v}2} \qquad &\text{on $\partial\mathbb{R}^2_+ $,}
\end{array}\right. \qquad  \int_{\R^2_+} e^v + \int_{\partial \R^2_+} e^{v/2} < +\infty.
\end{equation}
It is known (see \cite{li-zhu, Zhang}) that \eqref{limitproblem} is solvable for any $K_0 >0$, and also if $K_0 \leq 0$ and $h_0>\sqrt{-K_0}$. In any case, any solution is given by the expression:
\beq\label{entirehp}
v(w_1,w_2)= 2 \log \left\{ \frac{2 \lambda}{ K_0 \lambda^2 + (w_1-w_0)^2+(w_2+\lambda h_0)^2} \right\},\ \ w_0 \in \R, \ \lambda >0,
\eeq
that satisfies
\begin{equation} \label{quantization} h_0 \int_{\partial \R^2_+} e^{v/2} = \beta, \ \ \ K_0 \int_{\R^2_+} e^v = 2 \pi- \beta, \end{equation}
where $\beta$ is given by 
\beq\label{prebeta}
\displaystyle{\beta:=2\pi  \frac{h_0}{\sqrt{h_0^2+K_0}}  }.
\eeq

\subsection{Kazdan-Warner type conditions}
In this subsection we state and prove the Kazdan-Warner conditions for problem \eqref{gg}. There is a version already available for boundary problems in this way, posed in the half-sphere, see \cite{hamza}. To keep the paper self-contained we give our own version of it, which suits for our later purposes.

Let us recall the following Pohozaev-type identity, depending on an arbitrary field $F$.
\begin{lemma} \label{lemapoh} Let $u$ be a solution of \eqref{gg}. Then, given any vector field $F: \overline{ \D} \to \R^2$, there holds:
	\begin{equation*}\begin{split}
	\int_{\partial \D} &[2 K e^u ( F\cdot \nu) + (2 h e^{u/2} - 2) (\nabla u \cdot F) - \frac{|\nabla u|^2}{2} F \cdot \nu]  \\ 
	&= \int_{\D} [2 e^u ( \nabla K \cdot F + K\  \nabla \cdot F ) + DF(\nabla u, \nabla u)  - \nabla \cdot F \frac{|\nabla u|^2}{2}]. \end{split}\end{equation*}

\end{lemma}

\begin{proof} The proof follows by multiplying the equation \eqref{gg} by $\nabla u \cdot F$ and integrating by parts (see for instance \cite[Lemma 5.5]{lsmr}).
\end{proof}

From this we can obtain the following Kazdan-Warner identity.

\begin{proposition}\label{KW}
	If $u$ is a solution of \eqref{gg} then
	$$\int_{\D} e^u \, \nabla K \cdot F =4\int_{\partial\D}h_\tau e^{u/2}y,$$
	where $F(x,y):= (1-x^2+y^2, -2xy)$ (with complex notation, $F(z)= 1-z^2$).
\end{proposition}

\begin{proof}
	
The idea is to consider the variation along the conformal transformations (see \eqref{conformal group}) that keep fixed the point $p=(1,0)$. This corresponds to
	$$f_\lambda(z):=\frac{\lambda+z}{1+\lambda z},$$
	with $\lambda \in (-1,1)$.
	By taking into account the invariance of \eqref{limit}, we define 
	$$v_\lambda(z):=u(f_\lambda(z))+2\log (|f'_\lambda(z)|).$$
	In order to study its variation with respect to $\lambda$ we compute first
	$$F(z):=\frac{d}{d\lambda}f_\lambda(z)\bigg|_{\lambda=0}=\frac{1-z^2}{(1+\lambda z)^2}\bigg|_{\lambda=0}=1-z^2=(1-x^2+y^2,-2xy),$$
	where we have used the standard notation $z=x+iy$. From these identities we can conclude that $$\frac{d}{d\lambda}v_\lambda(z)\bigg|_{\lambda=0}=\nabla u\cdot F -4x.$$
	Hence, we test in the problem with $\nabla u\cdot F-4x$. We first multiply it by $\nabla u \cdot F $, which is just to insert $F$ in Lemma \ref{lemapoh}. Observe that on $\partial \D$, $F$ is a tangential vector field, and $F \cdot \tau = -2y$, where $\tau (x,y)= (-y,x)$. Moreover, since $F$ is holomorphic, the Cauchy-Riemann conditions yield
	$$ DF(\nabla u, \nabla u)  - \nabla \cdot F \frac{|\nabla u|^2}{2}=0.$$
Then, we obtain
\begin{equation} \label{unoh} \int_{\partial \D} (2 h e^{u/2}-2)(-2y) u_\tau = \int_{\D}  2 e^u ( \nabla K \cdot F - 4 x K ). \end{equation}
We now multiply \eqref{gg} by $4x$ and integrate to obtain
\begin{equation} \label{dos} \begin{split}   8 \int_{\D} K e^u x &= 4\int_{\D} u_x - 4\int_{\partial \D} x \nabla u \cdot \nu  = 4 \int_{\D} u_x  - 4 \int_{\partial \D} x (2 h e^{u/2} -2) \\
&= 4 \int_{\D} u_x + 4 \int_{\partial \D} y (2 h_\tau e^{u/2} + h e^{u/2} u_\tau).  \end{split} \end{equation}
Observe also that, integrating by parts,
\begin{equation} \label{tres} \int_{\partial \D} y u_\tau = - \int_{\partial \D} x u =  -\int_{\D} u_x.
\end{equation}
Putting together \eqref{unoh}, \eqref{dos} and \eqref{tres} we conclude.
\end{proof}
 
\begin{remark}  The same equality holds by interchanging the roles of $x$ and $y$, if we also change orientation $\tau \to -\tau$.
\end{remark}
 
\begin{remark} \label{obstruction} Let us observe that the Kazdan-Warner identity given in Proposition \ref{KW} can be written as
	$$ \int_{\D} e^u (1+x^2+y^2)^2\, \nabla K \cdot \nabla T  = -8 \int_{\partial\D}h_\tau T_\tau e^{u/2},$$	
where $T(x,y)= x/(1+x^2+y^2)$. As a consequence, if $\nabla K \cdot \nabla T$, $h_\tau T_\tau$ are both positive (or negative), then \eqref{gg} does not admit any solution. This is a typical Kazdan-Warner type obstruction; in the case of the semisphere, this was observed in \cite[Theorem 1]{hamza}. 
\end{remark}


\section{A blow--up analysis}\label{sec:blowup}
\setcounter{equation}{0}

In this section we begin the blow-up analysis given in Theorem \ref{main}. The main goal here is to give a proof of the Proposition \ref{mainprop} below.

Under the assumptions of Theorem \ref{main}, define:
\beq\label{xn}
u_n(\tilde{z}_n):=\max_{\overline{\D}} u_n(z)\to+\infty.
\eeq

We also define the singular set of the blowing-up sequence as in \cite{breme}:
\begin{equation}\label{singularset} 
\mathcal S:=\{p \in \oD : \ \exists \ z_n \in \oD, \ z_n \to p \, \, \, \mbox{ such that } \, \, \, u_n(z_n) \to +\infty\},
\end{equation}

By \eqref{xn}, up to a subsequence, we can assume that $\tilde{z}_n\to p \in \mathcal{S}$ as $n\to+\infty$.

The main result of this section is the following:

\begin{proposition}\label{mainprop}
	Under the assumptions of Theorem \ref{main}, $\mathcal{S}=\{p\} \subset \partial \D$ and, up to a subsequence, 
	$$
u_n\to -\infty \qquad \mbox{ locally uniform in } \oD\setminus \{p\}, 
$$ 
\beq\label{conver1}
h_ne^{u_n/2} \weakto \b \delta_p  \quad \mbox{ and } \quad K_ne^{u_n} \weakto (2\pi -\b) \delta_p,
\eeq
in the sense of measures, where
\beq\label{beta}
\displaystyle{\beta:=2\pi \frac{h(p)}{\sqrt{h^2(p)+K(p)}} }.
\eeq
Moreover,
\beq\label{diskprofile}
u_n(z)= 2 \log \left\{ \frac{2\phi_n(p)(1-|\zeta_n|^2)}{\phi^2_n(p)|1-\overline{\zeta_n}z|^2+K_n(p)|z-\zeta_n|^2} \right\} + O(1), 
\eeq
where $z\in\overline{\D}$, the function $\phi_n$ is defined in \eqref{phin} and $\zeta_n \in \D$, $\zeta_n\to p$.
\end{proposition}

This proposition is a first step in order to conclude the asymptotics of Theorem \ref{main}. In next section we will give a more accurate description (passing from $O(1)$ of \eqref{diskprofile} to $o(1)$). This will be needed in order to conclude the necessary conditions on the point $p$ in Section 5.

Proposition \ref{mainprop} will follow from the next result, which addresses more general Liouville-type equations:

\begin{proposition} \label{prop2} Let $u_n$ be a sequence of solutions of the problems
	\begin{equation} \label{ggn2}
	\left\{\begin{array}{ll}
	\displaystyle{-\Delta u_n + 2 \hat{K}_n= 2K_ne^{u_n} } \qquad & \text{in $\D$},\\
	\displaystyle{\frac{\partial u_n}{\partial \nu} +2 \hat{h}_n = 2h_ne^{u_n/2}} \qquad  &\text{on $\partial \D $},
	\end{array}\right.
	\end{equation}
	where $K_n \to K$, $\hat{K}_n \to \hat{K}$ in $C^2(\oD)$ and $h_n \to h$, $\hat{h}_n \to \hat{h}$ in $C^2(\partial \D)$ as $n\to +\infty$. Assume moreover $\sup \{ u_n \} \to +\infty$, but it has bounded global mass:
	
	\begin{equation} \label{mass2} \int_{\D} e^{u_n} \, +  \int_{\partial \D} e^{u_n/2} \leq C.	\end{equation}

	
	Then the singular set $\mathcal{S}$ is finite. Moreover,
	
	\begin{enumerate}
		\item[i)] If $p\in \mathcal{S}\setminus \partial\D$, then $K(p)>0$;
		\item[ii)] If $p\in\mathcal{S}\cap(\partial \D)$ and $K(p) \leq 0$, then $h(p) > \sqrt{-K(p)}$;
		
		\item[iii)] $u_n\to -\infty$ locally uniformly in $\oD\setminus \mathcal{S}$;

		\item[iv)] \beq\label{weak10}
K_n e^{u_n} \weakto \sum_{p\in \mathcal{S} \setminus (\partial \D)} 4\pi \delta_{p} + \sum_{p\in\mathcal{S} \cap \partial \D } \gamma_p \delta_{p},
\eeq
and
\beq\label{weak20}
h_n e^{u_n/2} \weakto \sum_{p\in\mathcal{S} \cap \partial \D } \gamma_p' \delta_{p},
\eeq
		where $\gamma_p+\gamma_p'=2\pi$ and \beq \label{beta2}	\gamma_p:=2\pi \frac{h(p)}{\sqrt{h^2(p)+K(p)}}.	\eeq
	\end{enumerate}
\end{proposition}

\begin{remark} Proposition \ref{prop2} holds true also for domains with boundary different from the disk, as well as for 2-manifolds with boudaries, with the obvious modifications in the proofs. However, since the goal of the paper is the study of the case of the disk, we have preferred a statement in this setting.
\end{remark}

First of all, let us introduce a minimal mass lemma, which is just a version for the case of boundaries of a well-known lemma by Brezis and Merle (\cite{breme}, see Theorem 1 and Corollary 3).

\begin{lemma}\label{minmass}
Under the assumptions of Proposition \ref{prop2}, let $p\in\D$ and $r>0$ such that $B_r(p) \subset \D$.
If
\beq\label{minmassi}
\displaystyle{ \int_{B_r(p)} K^+_n e^{u_n} \leq \varepsilon <2\pi},
\eeq
then $u_n^+$ is uniformly bounded in $L^{\infty}(B_{\frac{r}2}(p))$.

If, instead, $p \in \partial \D$ and
\beq\label{minmassb}
\int_{B_r(p) \cap \overline{\D}} K_n^+e^{u_n} \leq \varepsilon <\frac{\pi}{2}, \quad  \int_{B_r(p) \cap \partial \D } h_n^+ e^\frac{u_n}2\leq \varepsilon<\frac{\pi}{2},
\eeq
then $u_n^+$ is uniformly bounded in $L^{\infty}(B_{\frac{r}2}(p)\cap \overline{\D})$.
\end{lemma}
A proof of the previous result can be found in \cite[Lemma~2.4]{bls}.

\medskip 

As a consequence of this lemma, the set $\mathcal{S}$ (defined in \eqref{singularset}) must be finite. Moreover:
\beq\label{weak1}
K_n e^{u_n} \weakto \sum_{p \in \mathcal{S} \cap \D} \alpha_p \delta_{p} + \sum_{p\in\mathcal{S} \cap \partial \D } \gamma_p \delta_{p} +\Psi,
\eeq
and
\beq\label{weak2}
h_n e^{u_n/2} \weakto \sum_{p \in\mathcal{S} \cap \partial \D } \gamma_p' \delta_{p} + \Psi',
\eeq
where $\alpha_p\geq 2\pi$, $\gamma_p\geq \frac{\pi}{2}$ or $\gamma_p'\geq \frac{\pi}{2}$, $\Psi \in L^1(\D)\cap L^{\infty}_{loc}(\overline{\D}\setminus\mathcal{S})  ,\Psi'\in L^1(\partial \D)\cap L^{\infty}_{loc}(\partial\D\setminus\mathcal{S})$.

\medskip Taking into account \eqref{mass2}, for any point $p \in \mathcal{S} \cap \D$ we have:
\begin{equation*}\begin{split}
2\pi &\leq \int_{B_r(p)} K_n^+ e^{u_n} \leq  \int_{B_r(p)} \left( K_n^+(p) + r \|\nabla K_n^+\|_{L^\infty(B_r(p))}  \right) e^{u_n} \\
&\leq C(K_n^+(p) + r \|\nabla K_n^+\|_{L^\infty} ).
\end{split}\end{equation*}
By choosing $r>0$ sufficiently small we conclude that $K(p)>0$, so that $i)$ of Proposition \ref{prop2} holds.

\medskip

Let us now point out that, outside the set $\mathcal{S}$ we can assume that $e^{u_n}$ is uniformly bounded. Via a Green representation formula (or via local regularity estimates for the Neumann problem) one can deduce that $u_n$ has bounded oscillation far from $\mathcal{S}$: being more specific, 
\beq\label{oscill}
|u_n(z_1)-u_n(z_2)| \leq C \qquad \forall z_1,z_2 \in \displaystyle{ \overline{\D} \setminus \bigcup_{p\in\mathcal{S}} B_r(p)},
\eeq
for some $r>0$ fixed.

\medskip

\subsection{Proof of Proposition \ref{prop2}}

The proof of Proposition \ref{prop2} is the result of a more detailed study of the behavior of $u_n$ around the singular points. We shall focus on the boundary points, since the theory for interior blow-up points is much more developed (see for instance \cite{li, li-sha}).

\medskip 

Consider hence a point $p\in\mathcal{S} \cap \partial \D$ and a neighborhood of it which does not intersect any other singular point. Via a conformal map, we can pass to a problem in a half-ball. Let us be more specific, and assume without loss of generality that $p=(1,0)$. By the M\"obius transformation 
\[
\begin{array}{cccc}
f_1:&\mathbb{R}^2_+&\longrightarrow&\D\\
 &w&\longmapsto& z=\frac{i-w}{w+i}.
\end{array}
\]
we can map a semiball $B_1^+(0)\subset \mathbb{R}^2_+$ into the right half disk, where the point $p$ corresponds to the origin. Consider the transformation
$$
v_n(w):=u_n(f_1(w))+2\log|f'_1(w)|=u_n(f_1(w))+2\log \frac{2}{|w+i|^2}.
$$
Then $v_n$ satisfies the problem
\beq\label{problemplain}
\left\{\begin{array}{ll}
-\Delta v_n + 2 \hat{K}_n=2 {K}_n e^{v_n}&\text{ in }B_{r_0}^+(0),\vspace{.3cm}\\
\frac{\partial v_n}{\partial \nu} + 2 \hat{h}_n =2 h_n e^\frac{v_n}2&\text{ on $\G_{r_0}(0):=(-r_0,r_0)\times \{0\}$},
\end{array}\right.
\eeq
for some $r_0>0$ small, where $\mathcal{S} =\{0\}$. Here the functions $\hat{K}_n$, $\hat{h}_n$, $K_n$ and $h_n$ come from the original data composed with the transformation $f_1$: for the sake of clarity we keep the same notation. Besides, by \eqref{xn}, there exists a sequence $\{w_n\} \subset B_{r_0}^+(0)$ such that 
\beq\label{wn}
v_n(w_n)=\max_{\overline{B_{r_0}^+}(0)}v_n(w)\to +\infty,\quad w_n\to 0.
\eeq
Notice that $w_n$ may not match $f_1(\tilde{z}_n)$. 

%
%
%
%

\medskip

The following selection process has been applied many times in the literature starting from \cite{li-sha} (see in particular Lemma 4). The case with boundary have been treated in \cite[Lemma 7.4]{lsmr}, to which we refer for details (see also \cite{bjly}).

\begin{lemma} \label{lemilla} There exists a finite number of sequences $\Sigma_n :=\{w_1^n,\ldots,w_\ell^n, \hat{w}_1^n, \ldots \hat{w}_k^n \}\to 0$ and positive sequences $\varepsilon_1^n,\ldots,\varepsilon_\ell^n\to 0$, $\hat{\e}_1^n, \ldots \hat{\e}_k^n \to 0$ such that
\begin{enumerate}
\item[a)] $v_n(w_i^n)=\max \{ v_n(w), \ w \in B_{\varepsilon_i^n}(w_i^n) \cap \overline{\R^2_+} \} $ for $i = 1, \ldots, \ell$,
\item[b)] $\hat{\e}_i^n \leq dist(\hat{w}_i^n,\Gamma_{r_0}(0))$ and $v_n(\hat{w}_i^n)=\max \{ v_n(w), w \in B_{\hat{\varepsilon}_i^n}(\hat{w}_i^n)\} $ for $i = 1, \ldots, k$,
\item[c)] If $K(0)\leq 0$ then $k=0$ and $h(0) > \sqrt{-K(0)}$.
\item[d)] $\int_{B_{r}^+(0)} K_n e^{v_n} + \int_{ \Gamma_r(0)} h_n e^{v_n/2} \to 2 \pi \ell + 4 \pi k.$
\item[e)] $v_n \to - \infty$ uniformly in compact sets of $B_{r}^+(0) \setminus\{0\}$.
\end{enumerate}

\end{lemma}

The points $w_i^n$ represent points of local maxima for which the rescaled problem converge to a limit problem posed in a half-plane \eqref{limitproblem}, whereas for the points $\hat{w}_i^n$ the limit problem is posed in the entire plane \eqref{limit2}. The restrictions in $c)$ come from the obstructions to the existence of solutions for those problems. Moreover, the mass quantization in $d)$ comes from \eqref{quantization2} and \eqref{quantization}.

\begin{remark} In the analysis performed in \cite[Lemma 7.4]{lsmr} only the points $w_i^n$ appear. This is because in that paper $K$ is strictly negative and hence no solutions in the entire plane exist. However, the main point in the arguments of \cite{lsmr} is to show that the terms $\int_{B_{r}^+(0)} K_n e^{v_n}$, $\int_{ \Gamma_r(0)} h_n e^{v_n/2}$ do not give any contribution apart from that coming from their corresponding limit problems. That proof depends on the decay of the limit solutions \eqref{entirehp} and \eqref{entirehp2}, and works equally well in our framework.
\end{remark}

Let us point out that the coefficients $\gamma, \gamma'$ of the expansions \eqref{weak1} and \eqref{weak2} (we omit the dependence on $p$) satisfy:
\beq\label{gam1}
\gamma:=\lim_{r\to0}\lim_{n\to+\infty} \int_{B_{r}^+(0)} K_n e^{v_n}, 
\eeq
\beq\label{gam2}
\gamma':=\lim_{r\to0}\lim_{n\to+\infty} \int_{\Gamma_{r}(0) }  h_n e^{v_n/2}.
\eeq
Next we prove the quantization of this values using a proper Pohozaev type identity and the previous asymptotic estimates.

\begin{lemma}\label{quantiz}
Let $\gamma$ and $\gamma'$ given in \eqref{gam1} and \eqref{gam2}. Then
$$
\gamma+\gamma'=2\pi.
$$
\end{lemma}

\begin{proof}
Applying a Pohozaev type identity in $B_r^+(0)$ for $0<r<r_0$ with $r_0$ fixed in \eqref{problemplain}, as Lemma~\ref{lemapoh} with $F=w$, we have that
\begin{equation*}\begin{split}
	\int_{\partial B_r^+(0)} &[2  K_n e^{v_n} ( w \cdot \nu) + (2  h_n e^{v_n/2} - 2 \hat{h}_n) (\nabla v_n \cdot w) - \frac{|\nabla v_n|^2}{2} w \cdot \nu]\,dw  \\ 
	&= \int_{B_r^+(0)} [2 e^{v_n} ( \nabla {K}_n \cdot w + 2 {K}_n  ) + 2 \hat{K}_n \nabla v_n \cdot w]\,dw. 
\end{split}\end{equation*}
If we divide $\partial B_r^+= \partial^+ B_r \cup \Gamma_r$, we immediately obtain 
\begin{equation*}\begin{split}
&	\int_{\partial B_r^+(0)} [2  K_n e^{v_n} ( w \cdot \nu) + (2  h_n e^{v_n/2} - 2 \hat{h}_n) (\nabla v_n \cdot w) - \frac{|\nabla v_n|^2}{2} w \cdot \nu]\,dw \\
	&\;=	\int_{\partial^+ B_r (0)} [2 r   K_n e^{v_n}  + (2  h_n e^{v_n/2} - 2 \hat{h}_n) (\nabla v_n \cdot w) - \frac{|\nabla v_n|^2}{2} r ]\,dw
\\
&\quad +\int_{ \Gamma_r(0)} (2  h_n e^{v_n/2} - 2\hat{h}_n) \partial_1 v_n w_1\,dw,
\end{split}\end{equation*}
where we have used the notation $w=w_1+iw_2$. Taking into account that $|\nabla  K_n|\leq C$ uniformly in $B_r^+(0)$, then
		\beq\label{quant1}
		\int_{B_r^+(0)} (\nabla  K_n \cdot w)e^{v_n}\,dw \leq r \int_{B_r^+(0)} |\nabla K_n| e^{v_n}\,dw = O(r).
		\eeq
		On the other hand, integrating by parts
\begin{equation*}\begin{split}
		\int_{\Gamma_r(0)}2  h_n e^{v_n/2} \nabla v_n \cdot w\,dw =&\, 4\Big[  h_n e^{v_n/2} w_1 \Big]_{w=(-r,0)}^{(r,0)} -4\int_{\Gamma_r(0)} \nabla  h_n \cdot w \, e^{v_n/2}\, dw\\
		& - 4\int_{\Gamma_r(0)}  h_n e^{v_n/2}\,dw.
\end{split}\end{equation*}
		Again, since $|\nabla  h_n | \leq C$, we have
		\beq\label{quant2}
		\int_{\Gamma_r(0)} \nabla  h_n \cdot w \, e^{v_n/2} \,dw= O(r).
		\eeq
		Moreover, by Lemma \ref{lemilla}, $e)$,  we deduce that
		\beq\label{quant3}
		\Big[  h_n e^{v_n/2} w_1 \Big]_{w=(-r,0)}^{(r,0)} \to 0, \quad \int_{\partial^+B_r(0)}  K_ne^{v_n}r\,dw  \to 0, \quad \mbox{ as $n\to+\infty$}.
		\eeq
		By Green representation formula,
\begin{equation*}\begin{split}
\nabla v_n(w)=& \frac{1}{\pi} \int_{B_r^+(0)} \frac{w-z}{|w-z|^2} 2 K_n(z) e^{v_n(z)} dz \\
&+ \frac{1}{\pi} \int_{\Gamma_r(0)} \frac{w-z}{|w-z|^2} 2 h_n(z) e^{v_n(z)/2} dz+O(r),
\end{split}\end{equation*}
and, using \eqref{weak1} and \eqref{weak2},
$$
\nabla v_n(w)\to \nabla \mathcal{G} = \frac 2\pi (\gamma+\gamma')\frac{w}{|w|^2}+O(r) \mbox{ outside the origin,}
$$
		$$
		\int_{\partial^+ B_r(0)} \nabla v_n \cdot w\,dw \to \int_{\partial ^+B_r(0)} \nabla \mathcal{G} \cdot w\,dw = O(r), 
		$$
		
		$$
		\int_{B_r(0)} \nabla v_n \cdot w\,dw \to \int_{B_r(0)} \nabla \mathcal{G} \cdot w\,dw = O(r), 
		$$
		$$
		\int_{\partial^+ B_r(0)} \frac{|\nabla v_n|^2}{2} r\,dw  \to \int_{\partial^+ B_r(0)} \frac{|\nabla \mathcal{G}|^2}{2} r\,dw  = 2(\gamma+\gamma')^2 \frac{1}{\pi}+O(r). 
		$$
Plugging the expressions above we arrive at 
$$
2(\gamma+\gamma')^2 \frac{1}{\pi} = 4(\gamma+\gamma') +O(r),
$$
giving the desired conclusion.
\end{proof}

Observe that the previous lemma, together with Lemma \ref{lemilla}, $d)$, implies that $\ell =1 $ and $k=0$. Moreover, by \eqref{prebeta},

$$ \gamma = 2 \pi \frac{h(0)}{h^2(0)+K(0)}.$$ 

This finishes the proof of Proposition \ref{prop2}.

\subsection{Proof of Proposition \ref{mainprop}}

We finally turn our attention back to problem \eqref{ggn}. In this specific case we first show that the singular set is formed by a unique point located at $\partial \D$.

\begin{lemma}
The blow--up set $S=\{p\}\subset \partial \D$.
\end{lemma}

\begin{proof}
By \eqref{GB} we know
$$\int_{\D} K_n e^{u_n} + \int_{\partial \D} h_n e^{u_n/2}=2\pi.$$ Furthermore, Proposition \ref{prop2} implies that
$$\int_{\D} K_n e^{u_n} + \int_{\partial \D} h_n e^{u_n/2}\to 4\pi N+2\pi M,$$ 
with $N = |\mathcal{S} \cap \D| $ and $M = |\mathcal{S} \cap \partial \D|$, meaning the cardinals of those sets. Thus, the only possibility is $N=0$, $M=1$.
\end{proof}

Next, we establish global pointwise estimates of the bubbling sequence around $0$. This type of estimate was first derived by Y.Y. Li in \cite{li} by the method of moving planes. Another argument was given by Bartolucci-Chen-Lin-Tarantello in \cite{bclt} for singular problems, which is more suited to our framework (see also Section~4.2. in \cite{Tarhb} for details). Since this argument is rather standard for Liouville's type problems, we will be sketchy.

\begin{lemma}\label{localprof}
	If $K(p) \leq 0$, then $h(p) > \sqrt{-K(p)}$. Moreover, the solutions $v_n$ of \eqref{problemplain} satisfy:
	
	\beq\label{loc}
	v_n(w_1,w_2)= 2 \log \left\{ \frac{2 \lambda_n}{ K_n(0)\lambda^2_n + (w_1-w_{1,n})^2+(w_2-w_{2,n}+ h_n(0)\lambda_n)^2} \right\} + O(1),
	\eeq
	in $B^+_{r_0}(0)$ where $w_n=(w_{1,n},w_{2,n})\to 0$ defined in \eqref{wn} and:
\beq\label{lambdan}
	\lambda_n:=\frac{2\delta_n}{{K}_n(0)+{h}^2_n(0)} \quad \mbox{ with } \quad {\delta}_n:=e^{-\frac{v_n(w_n)}{2}}.
\eeq

\end{lemma}

\begin{proof}
	Define the rescaled function
	$$
	\tilde{v}_n(w):=v_n(w_n+ \delta_n w)+2\log  \delta_n,
	$$
	that satisfies the problem
	\beq\label{problempresc}
	\left\{\begin{array}{ll}
		-\Delta \tilde{v}_n=2 {K}_n(w_n+\delta_n w)e^{\tilde{v}_n}&\text{ in } B_{{r_0}/\delta_n}^+\left(-\frac{w_n}{\delta_n}\right),\vspace{0.2cm} \\ 
		\frac{\partial \tilde{v}_n}{\partial \nu} =2 {h}_n(w_n+\delta_n w)e^\frac{\tilde{v}_n}2&\text{on } \G_{{r_0}/\delta_n}\left(-\frac{w_n}{\delta_n}\right),
	\end{array}\right.
	\eeq
	where
	
	$$ B_{{r_0}/\delta_n}^+\left(-\frac{w_n}{\delta_n}\right):=\left\{(w_1,w_2)\in B_{{r_0}/\delta_n}\left(-\frac{w_n}{\delta_n}\right): w_2\geq -\frac{w_n}{\delta_n}\right\}, $$
	$$ \G_{{r_0}/\delta_n}\left(-\frac{w_n}{\delta_n}\right):=(-{r_0}/\delta_n,{r_0}/\delta_n)\times \left\{-\frac{w_n}{\delta_n}\right\}.$$
	By simplicity we assume $-\frac{w_n}{\delta_n}\to q=0$. Thus, by the Harnack inequality (see for instance \cite{jwz}):
	\beq\label{uno}
	\tilde{v}_n \, \mbox{ is uniformly bounded in } L^\infty_{loc}(\mathbb{R}_+^2),
	\eeq
	and
	\beq\label{due}
	\tilde{v}_n \to v_0 \qquad \mbox{ in } C^2_{loc}(\mathbb{R}_+^2),
	\eeq
	where $v_0$ is the entire solution of the problem
	\beq\label{problempresc2}
	\left\{\begin{array}{ll}
		-\Delta v_0=2  K(0)e^{v_0}&\text{ in }\mathbb{R}^2_+,\vspace{.3cm}\\
		\frac{\partial v_0}{\partial \nu} =2  h(0)e^\frac{v_0}2&\text{on $\partial \mathbb{R}^2_+$,} \vspace{.3cm} \\
		v_0(0)=0, &
	\end{array}\right.
	\eeq
	with
	$$\int_{\mathbb{R}^2_+} e^{v_0} + \int_{\partial \mathbb{R}^2_+} e^{v_0/2} <C.$$
	According to Subsection \ref{sublimit}, if $K(0) \leq 0$, then $h(0) > \sqrt{-K(0)}$. Moreover, $v_0$ takes the form:
	\beq\label{locR2+}
	v_0(w_1,w_2)= 2 \log \left\{ \frac{2 \lambda_0}{  K(0) \lambda_0^2 + w_1^2+(w_2+\lambda_0  h(0))^2} \right\},
	\eeq
	where $\lambda_0:=\frac{2}{ K(0)+ h(0)^2}>0$. In addition, we have that
	\beq\label{profile2bis}
	|w_n| \leq C \delta_n.
	\eeq
	
	\medskip

	We are now concerned with the global $O(1)$ estimate. By the analysis in \cite[Lemma 7.4]{lsmr} (see in particular Step 4), we conclude that for every $\varepsilon$ there exists $R_\varepsilon>1$ such that
	\beq\label{deca}
	\tilde v_n(w)\leq (4-\varepsilon)\log |w|+C_\varepsilon, \qquad \mbox{ for $|w|\geq 2R_\varepsilon$}.
	\eeq
		
Taking into account  \eqref{oscill}	and \eqref{due}, by Green representation formula one has that 
\begin{equation*}\begin{split}
\tilde{v}_n(w)+\frac{M_n}{\pi} \log |w|&= \frac{1}{\pi} \int_{B^+_{r_0/\delta_n}\left(-\frac{w_n}{\delta_n}\right)} \log \left( \frac{|w||z|}{|w-z|}\right) 2 K_n(w_n+\delta_n z) e^{\tilde v_n(z)}\,   dz  \\
&\;+  \frac{1}{\pi} \int_{\Gamma_{r_0/\delta_n}\left(-\frac{w_n}{\delta_n}\right)} \log \left( \frac{|w||z|}{|w-z|}\right) 2 h_n(w_n+\delta_n z) e^{ \frac{\tilde v_n(z)}{2}}\,   dz  + \Psi_n,
\end{split}\end{equation*}
where $\Psi_n$ is uniformly bounded in $\overline{B^+_{r_0/\delta_n}}\left(-\frac{w_n}{\delta_n}\right)$ and $M_n$ is defined as
\begin{equation*}\begin{split}
	M_n&:=\int_{B^+_{r_0}(0)}  2 K_ne^{v_n}+ \int_{\Gamma_{r_0}(0)}  2 h_n e^{\frac{v_n}{2}}\\
	&= \int_{B^+_{r_0/\delta_n}\left(-\frac{w_n}{\delta_n}\right)}  2{K}_n(w_n+\delta_n w)e^{\tilde{v}_n}+ \int_{\Gamma_{r_0/\delta_n}\left(-\frac{w_n}{\delta_n}\right)}  2{h}_n(w_n+\delta_n w)e^\frac{\tilde{v}_n}2.
\end{split}\end{equation*}
	We can estimate the expression above using the asymptotic behavior given by \eqref{deca} in order to conclude that
	\beq\label{tre}
	\left|\tilde{v}_n(w)+\frac{M_n}{\pi} \log |w| \right| = O(1) \quad \mbox{ for } 3\leq |w| \leq \frac{r_0}{ \delta_n}.
	\eeq
	See  Lemma~4.2.4. in \cite{Tarhb} for further details on the computations of the logarithmic terms.

Finally, using a Pohozaev type identity in $B^+_{-\delta_n\log\delta_n}(0)$ one can obtain that
	\beq\label{quatro}
	\left|M_n - 4\pi \right|=O\left( |\log\delta_n |^{-1}\right).
	\eeq
	We refer the reader to Lemma~4.2.5. in \cite{Tarhb} for more details.

\
	
	Combining \eqref{uno}, \eqref{due}, \eqref{tre} and \eqref{quatro}, we conclude
	$$
	\left| \tilde v_n(w_1,w_2) - 2 \log \left\{ \frac{2 \lambda_0}{ K(0)\lambda_0^2 + w_1^2+(w_2+ h(0)\lambda_0)^2} \right\} \right| \leq C, \quad \mbox{ in } \overline{B^+_{r_0/\delta_n}}(0).
	$$
	Using the definition of $\tilde v_n$ and \eqref{wn}, scaling back to $v_n$ we obtain \eqref{loc}.	 
\end{proof}

\medskip We now conclude the proof of Proposition \ref{mainprop} by showing the estimate \eqref{diskprofile}.
We transform properly the solution of \eqref{problemplain} into a solution in the set 
$$\mathcal{D}:= \left\{z=x+iy \in \overline{\D} : x>\frac{1-r_0}{1+r_0} \right\}.$$
In order to do that, note that
\begin{equation*}\begin{split}
&u_n(z)=v_n(f_1^{-1}(z))+2\log\frac{2}{|z+1|^2}\\
\end{split}\end{equation*}
where $f_1^{-1}$ is a M\"obius transformation, inverse of $f_1$, namely
\[
\begin{array}{cccc}
f_1^{-1}:&\overline{\D}&\longrightarrow&\mathbb{R}_+^2 \vspace{0.2cm}\\
 &z&\longmapsto& w=i\frac{1-z}{z+1},
\end{array}
\]
that is,
$$u_n(z)=2\log\left\{\frac{4\lambda_n}{\left[{K}_n(0)\lambda_n^2+\left(\frac{2y}{|z+1|^2}-w_{1,n}\right)^2+\left(\frac{1-|z|^2}{|z+1|^2}-w_{2,n}+{h}_n(0)\lambda_n\right)^2\right]|z+1|^2}\right\}.$$
To complete the proof, define a sequence $\zeta_n$ as
$$
\zeta_n:=\left(1-\lambda_n \frac{\phi_n^2(p)+K_n(p)}{\phi_n(p)}\right) p .
$$
Taking into account \eqref{lambdan}, observe that $\zeta_n\to p$ as $n\to+\infty$ and $|\zeta_n|<1$. Moreover, \eqref{profile2bis} holds by the choice of $\zeta_n$. Considering the function
\beq\label{profilenhdisk}
\tilde{u}_n(z):= 2 \log \left\{ \frac{2\phi_n(p)(1-|\zeta_n|^2)}{\phi^2_n(p)|1-\overline{\zeta_n}z|^2+K_n(p)|z-\zeta_n|^2} \right\},
\eeq
where $z\in \mathcal{D}$, it can be proved that
$$
|\tilde{u}_n(z)-u_n(z)|=O(1),
$$
that is, 
$$u_n(z)=\tilde{u}_n(z)+O(1), \quad z\in \mathcal{D}.$$
We finally extend the previous expression to the whole disk. Letting $z\in \D$ and $z_1 \in \mathcal{D}$, by \eqref{oscill} and \eqref{profilenhdisk} we obtain that
\begin{equation*}\begin{split}
&\left| u_n(z) - 2 \log \left\{ \frac{2\phi_n(p)(1-|\zeta_n|^2)}{\phi^2_n(p)|1-\overline{\zeta_n}z|^2+K_n(p)|z-\zeta_n|^2} \right\} \right| \\
&\leq\left| u_n(z_1) - 2 \log \left\{ \frac{2\phi_n(p)(1-|\zeta_n|^2)}{\phi^2_n(p)|1-\overline{\zeta_n}z|^2+K_n(p)|z-\zeta_n|^2} \right\} \right| + \left| u_n(z)- u_n(z_1) \right |\\
&\leq C,
\end{split}\end{equation*}
that gives the desired conclusion.

\section{Estimate of the error term}\label{error}
\setcounter{equation}{0}

%
%
%
%
%
%
%
%
%
%
%

\begin{proposition} \label{quantitative} Under the assumptions of Theorem \ref{main}, there exist $a_n \in  \D$ such that, defining
$$v_n(z):= u_n(f_{a_n}(z)) + 2 \log \frac{1-|a_n|^2}{|1 + \overline{a_n} z |^2},$$
with $f_a(z):= \frac{a+z}{1 + \overline{a} z}$, there holds:
\begin{enumerate}
	\item[a)] $v_n$ solves the problem
	\begin{equation} \label{eq-v}
	\left\{\begin{array}{ll}
	\displaystyle{-\Delta v_n = 2K_n(f_{a_n}(z))e^{v_n(z)} } \qquad & \text{in $\D$},\\
	\displaystyle{\frac{\partial v_n}{\partial \nu} +2 = 2h_n(f_{a_n}(z))e^{v_n(z)/2}} \qquad  &\text{on $\partial \D $},
	\end{array}\right.
	\end{equation}
	
	\item[b)] $\int_{\D} x e^{v_n(z)} \, dz = \int_{\D} y e^{v_n(z)} \, dz =0$,
	\item[c)] $a_n \to p$ where $p \in \partial \D$ is the blow-up point,
	\item[d)] $v_n$ is uniformly bounded,
	\item[e)] $v_n(z) = 2 \log \left( \frac{2 \hat{\phi}_n}{\hat{\phi}_n^2 + \hat{k}_n |z|^2} \right) + \xi_n(z)$,
	where $$ \hat{\phi}_n:= \phi_n \left(\frac{a_n}{|a_n|} \right), \hat{k}_n:=K_n \left(\frac{a_n}{|a_n|} \right),$$
	 with $\phi_n$ defined in \eqref{phin} and $\| \xi_n \|_{C^{0,\alpha}(\overline{\D})} \leq C (1-|a_n|)^{1-\alpha}$ for any  $\alpha \in (0,1/2)$. 
\end{enumerate}
\end{proposition}

\begin{proof} Assertion a) is immediate. In order to prove b), given any continuous function $u$ we define $ \Gamma : \D \to \R^2$ as
$$\Gamma(a) := \int_{\D} z e^{v_a(z)} \, dz,$$
where
$$v_a(z):= u(f_{a}(z)) + 2 \log |f_a'(z)|.$$
We claim that there exists $a \in \D$ such that $\Gamma(a)=0$.
Observe that if $a_n \to a \in \partial \D$,
\begin{equation*}\begin{split}
\int_{\D} z e^{v_{a_n}(z)} \, dz &= \int_{\D} z e^{u(f_{a_n}(z))} |f_{a_n}'(z)|^2 \, dz  = \int_{\D} f_{-{a_n}}(z') e^{u(z')} \, dz' \\
&\to -a \int_{\D} e^{u(z')} dz',
\end{split}\end{equation*}
by Lebesgue Theorem. Then, the map $\Gamma$ can be extended in a continuous way to the boundary by $\Gamma(a)= -a \int_{\D} e^{u(z')} \, dz'$. As a consequence its Brouwer degree is $1$ and we conclude the claim.

We now prove $c)$. Assume by contradiction that, up to a subsequence, $a_n \to a_0 \in \overline{\D}$, $a_0 \neq p$. Observe that $f_{a_n}$ converges to $f_{a_0}$ uniformly in $B_p(r)$, for small $r>0$. Recall that by Proposition \ref{mainprop}, $K_n e^{u_n} \weakto \beta \delta_p$, with $p \in \partial \D$. Then
$$0 = \int_{\D} z e^{v_n(z)} \, dz =  \int_{\D} z e^{u_n(f_{a_n}(z))} |f_{a_n}'(z)|^2 \, dz  = \int_{\D} f_{-a_n}(z) e^{u_n(z)} \, dz \to c_0\frac{-a_0+p}{1-\overline{a_0}p},$$
with $c_0\neq 0$ a constant independent of $a$, a contradiction.

Hence, let $a_n\to p$. Taking $u_n$ as in Proposition~\ref{mainprop}, define 
$$
v_n(z):=u_n(f_{a_n}(z)) +2\log \frac{1-|a_n|^2}{|1+\overline{a_n}z|^2}, 
$$
which implies
$$
v_n(z)= 2 \log \left\{ \frac{2\phi_n(p)(1-|\zeta_n|^2)(1-|a_n|^2)}{\phi^2_n(p)|1-\overline{\zeta_n}a_n+z(\overline{a_n}-\overline{\zeta_n})|^2+K_n(p)|a_n-\zeta_n+z(1-\overline{a_n}\zeta_n)|^2} \right\}+ O(1).
$$
Letting $\tilde{a}_n=\frac{a_n-\zeta_n}{1-\overline{a_n}\zeta_n} \in \D $ we can rewrite the previous expression as
$$
v_n(z)= 2 \log \left\{ \frac{2\phi_n(p)(1-|\tilde{a}_n|^2)}{\phi^2_n(p)|1- \overline{\tilde{a}_n} z|^2+K_n(p)|z-\tilde{a}_n|^2} \right\}+ O(1).
$$
Suppose that $\tilde{a}_n\to\tilde{a} \in \partial \D$. Then $v_n(\tilde a_n)\to +\infty$ and $e^{v_n}\to 0$ locally in $\D\setminus \{\tilde{a}\}$. However, $v_n$ has barycenter $0$, which contradicts $\tilde{a} \in \partial \D$. Consequently, $|\tilde{a}_n|<1-\varepsilon$ uniformly for $\varepsilon>0$ and $v_n$ remains uniformly bounded, what proves $d)$.

We finally prove $e)$. By $d)$ $v_n$ is bounded from above and hence the terms
$$2K_n(f_{a_n}(z))e^{v_n(z)} , \ \ 2h_n(f_{a_n}(z))e^{v_n(z)/2},$$
are bounded in $L^\infty$. By Lemma \ref{regularity}, $v_n$ is bounded in $W^{1+1/q, q}(\D)$ for all $q>1$. By the Sobolev embeddings, $v_n$ is bounded in $C^{0,\alpha}(\overline{\D})$ for all $\alpha \in (0,1)$ and thus, up to a subsequence, we can assume that $v_n \to v_0$ in $C^{0,\alpha}(\overline{\D})$, where $v_0$ is a solution of 
	\begin{equation*} 
\left\{\begin{array}{ll}
\displaystyle{-\Delta v_0 = 2K(p)e^{v_0(z)} } \qquad & \text{in $\D$},\\
\displaystyle{\frac{\partial v_0}{\partial \nu} +2 = 2h(p)e^{v_0(z)/2}} \qquad  &\text{on $\partial \D $.}
\end{array}\right.
\end{equation*}
Furthermore, $\int_{\D} x e^{v_0(z)} \, dz = 0$, $\int_{\D} y e^{v_0(z)} \, dz =0$. Applying Remark \ref{rem rad} we conclude that $v_0$ is given by \eqref{radialprofile}.

Denote
$$ \tilde{v}_n(z):= 2 \log \left( \frac{2 \hat{\phi}_n}{\hat{\phi}_n^2 + \hat{k}_n |z|^2} \right) .$$
Thus, $\xi_n(z)= v_n(z) - \tilde{v}_n(z)$, and it converges to $0$ in $C^{0,\alpha}(\overline{\D})$ sense, for any $\alpha \in (0,1)$, by the arguments above. Our aim is to give a quantitative estimate on this convergence.
Observe that $\xi_n$ satisfies:
\begin{equation*} 
\left\{ \begin{array}{ll}
-\Delta \xi_n =  2  K_n\left(f_{a_n}\left(\frac{a_n}{|a_n|}\right)\right) e^{\tilde{v}_n} \xi_n + c_n(z)
& \mbox{in }\D, \\
\frac{\partial \xi_n}{\partial \nu}  = 2 h_n \left( f_{a_n}\left(\frac{a_n}{|a_n|}\right) \right) e^{\tilde{v}_n/2} \xi_n + d_n(z), 
&\text{on } \partial \D.
\end{array} \right.
\end{equation*}
with
\begin{equation*}\begin{split}
c_n(z):=&\, 2  \left [ K_n(f_{a_n}(z)) - K_n \left ( f_{a_n}\left(\frac{a_n}{|a_n|}\right) \right ) \right ] e^{v_n}  \\
&+ 2  K_n \left (f_{a_n}\left(\frac{a_n}{|a_n|}\right) \right ) e^{\tilde{v}_n} (e^{\xi_n}-1 - \xi_n),\\
 d_n(z):=& \,2 \left [ h_n(f_{a_n}(z)) - h_n\left (f_{a_n}\left(\frac{a_n}{|a_n|}\right) \right ) \right ] e^{v_n/2}  \\
 &+ 2  h_n\left (f_{a_n}\left(\frac{a_n}{|a_n|}\right) \right ) e^{\tilde{v}_n/2} (e^{\xi_n/2}-1 - \xi_n/2 ). 
 \end{split}\end{equation*}
By the mean value theorem,
$$\left | K_n(f_{a_n}(z)) - K_n \left ( f_{a_n}\left(\frac{a_n}{|a_n|}\right) \right ) \right | \leq C \left | f_{a_n}(z) - f_{a_n}\left(\frac{a_n}{|a_n|}\right) \right |,$$
$$\left | h_n(f_{a_n}(z)) - h_n \left ( f_{a_n}\left(\frac{a_n}{|a_n|}\right) \right ) \right | \leq C \left | f_{a_n}(z) - f_{a_n}\left(\frac{a_n}{|a_n|}\right) \right |,$$
and applying Proposition \ref{appendix} we obtain
$$ \int_{\D} \left | f_{a_n}(z) - f_{a_n}\left(\frac{a_n}{|a_n|}\right) \right |^q \leq C (1-|a_n|)^q, \ \ \int_{\partial \D} \left | f_{a_n}(z) - f_{a_n}\left(\frac{a_n}{|a_n|}\right) \right |^q \leq C (1-|a_n|),$$
for any $1<q<2$.

Since $\xi_n \to 0$ uniformly, we have that $\| e^{\xi_n } - 1 - \xi_n \|_{L^q} \leq C \| \xi_n \|_{L^{\infty}} \| \xi_n\|_{L^q}$. We now apply Lemma \ref{lema linear 2}, to conclude that
\begin{equation} \label{lollo} \| \xi_n \|_{W^{1+1/q,q}(\D)} \leq  C_q (1-|a_n|)^{1/q} + o_n(1) (\| \xi_n\|_{L^q(\D)}+\| \xi_n\|_{L^q(\partial\D)}). \end{equation}
On the other hand, we have that $\| \xi_n \|_{L^{q}(\D)}+\| \xi_n\|_{L^q(\partial\D)} \leq C\| \xi_n \|_{W^{1+1/q,q}(\D)}$, from which we obtain, making $n$ large enough, that
%
%
$$ \| \xi_n \|_{W^{1+1/q,q}(\D)} \leq  (C_q+2) (1-|a_n|)^{1/q}.$$
Now it suffices to recall the continuous Sobolev embedding $W^{1+1/q, q}(\D) \subset C^{0,\alpha}(\overline{\D})$ for $\alpha = 1-1/q$, to conclude the proof of Proposition \ref{quantitative}.
\end{proof}

\subsection{Proof of Theorem \ref{main}, $i)$ and $ii)$}

From the above analysis, $i)$ readily follows. Moreover, with the notation of Proposition \ref{quantitative} it is enough to define 
$$u_n(z):= v_n(f_{-a_n}(z)) - 2 \log \frac{1-|a_n|^2}{|1 +\overline{a_n} z |^2}.$$
to obtain $ii)$. Observe also that
$$ \psi_n(z) = \xi_n(f_{-a_n}(z)), $$
and $|f_{-a_n}'(z)| \leq C (1-|a_n|)$. Hence, the mean value theorem allows us to estimate the $C^{0,\alpha}$ norm of $\psi_n$ as
\begin{equation*}\begin{split}
 |\psi_n(z_1) - \psi_n(z_2)| &= |\xi_n(f_{-a_n}(z_1)) - \xi_n(f_{-a_n}(z_2)) | \\
 &\leq C (1-|a_n|)^{1-\alpha} |f_{-a_n}(z_1) - f_{-a_n}(z_2)|^{\alpha} \\
 &\leq C (1-|a_n|)^{1-2\alpha}|z_1-z_2|^{\alpha}.
 \end{split}\end{equation*}

\section{Conclusion of the proof of Theorem \ref{main}} \label{sec:proof}
\setcounter{equation}{0}

In this section we conclude the proof of the main Theorem \ref{main} by showing $iii)$. 
Multiplying \eqref{ggn} by a proper vector field and integrating drives to a first blow--up condition. 

\begin{proposition}\label{firstcond}
Let $p \in S$ and recall the definition of $\Phi$ given in \ref{PHI}. Then
\beq\label{realcond}
2 h_\tau(p) + \frac{K_\tau(p)}{\Phi(p)} =0,
\eeq
or, equivalently, $\Phi_\tau(p)=0$.
\end{proposition}

\begin{proof}
Consider the vector field $F:\overline{\D}\to\mathbb{R}^2$ defined by $F(x,y)=(-y,x)$. Observe that $F$ is the tangential vector in $\partial\D$. Applying Lemma~\ref{lemapoh} and integrating by parts we obtain
	\beq\label{Poh1}
\displaystyle{ -2  \int_{\partial \D} (u_n)_\tau  = 4 \int_{\partial \D} (h_n)_\tau e^{u_n/2}  +  2 \int_{ \D } \nabla K_n \cdot F e^{u_n}},
  	\eeq
where $u_n$ is the solution of \eqref{ggn} satisfying Theorem \ref{main}, $i)$, $ii)$.
Here the lower index $\tau$ means the tangential derivative, i.e. in the direction of $F$. Observe that  we have used the facts that $\nabla \cdot F=0$ and $DF (\nabla u_n,\nabla u_n)=0$.

Reasoning similarly as we did to obtain the distributional convergences \eqref{conver1} and using the quantization of mass given in Proposition \ref{mainprop}, we immediately get that
$$
(h_n)_\tau e^{u_n/2} \weakto  2\pi \frac{h_\tau(p)}{\sqrt{h^2(p)+K(p)}} \delta_p,
$$
$$
\Phi_n e^{u_n} \weakto 2\pi \frac{1}{\sqrt{h^2(p)+K(p)}} \delta_p,
$$
where the function $\Phi$ is defined in \eqref{PHI}.

Multiplying and dividing properly by $\Phi_n$, by the above convergences and the pointwise estimates on $u_n$ the right-hand side in the identity \eqref{Poh1} converges to
	\begin{equation}\label{eq:ktht}
\displaystyle{	8\pi  \frac{h_\tau(p)}{\sqrt{h^2(p)+K(p)}} + 4\pi \frac{K_\tau(p) }{\Phi(p)\sqrt{h^2(p)+K(p)}}}.
	\end{equation}
Noticing that
$$
\int_{\partial \D}  (u_n)_\tau  =0  , 
	$$
the expression \eqref{eq:ktht} vanishes and this implies \eqref{firstcond}.
\end{proof} 

Without loss of generality from now on we will suppose $p=(1,0)$. We  can also assume, by composing with suitable rotations of the functions $K_n$, $h_n$, that $\lambda_n\in \R$, $\lambda_n \to 1$. Consider
\begin{equation}\begin{split}\label{uLambda}
u_{\lambda_n}(x,y):=&2\log \left\{\frac{2(1-\lambda_n^2)\hat{\phi}_n}{\hat{\phi}_n^2(1-\lambda_n x)^2+\hat{\phi}_n^2(\lambda_n y)^2+\hat{k}_n(x-\lambda_n)^2+\hat{k}_ny^2}\right\},
\end{split}\end{equation}
that corresponds to the profile given in Theorem \ref{main}.
Notice that the denominator of the logarithm is positive as a consequence of the fact 
$$\Phi(p)^2+K(p)>0.$$

\begin{lemma}\label{asint}
If $u_n=u_{\lambda_n}+\psi_n$ with $u_{\lambda_n}$ given by \eqref{uLambda} and $\|\psi_n\|_{C^{0,\alpha}}\leq C(1-\lambda_n)^{1-2\alpha}$, $\alpha\in (0,1/2)$, then
\begin{itemize}
\item[(i)] $\displaystyle \int_{\partial\D}(h_n)_\tau e^{\frac{u_n}{2}}y=e^{\frac{\psi_n(p)}{2}}\int_{\partial\D}(h_n)_\tau e^{\frac{u_{\lambda_n}}{2}}y+o_n(1)(1-\lambda_n).$

\item[(ii)] $\displaystyle  \int_{\D}\left((K_n)_y xy-(K_n)_x(1-x^2+y^2)\right)e^{u_n}$

$\displaystyle \; =e^{\psi_n(p)}\int_{D^2}\left((K_n)_y xy-(K_n)_x(1-x^2+y^2)\right)e^{u_{\lambda_n}}+o_n(1)(1-\lambda_n),$
\end{itemize}
where $o_n(1)$ is a quantity that goes to zero as $n\rightarrow\infty$.
\end{lemma}

\begin{proof}
By definition
$$e^{\frac{u_n}{2}}=e^{\frac{u_{\lambda_n}+\psi_n}{2}}=e^{\frac{\psi_n(p)}{2}}e^{\frac{u_{\lambda_n}}{2}}e^{\frac{\psi_n-\psi_n(p)}{2}}=e^{\frac{\psi_n(p)}{2}}e^{\frac{u_{\lambda_n}}{2}}(1+O(\psi_n-\psi_n(p))),$$
and hence
$$\int_{\partial\D}(h_n)_\tau e^{\frac{u_n}{2}}y=e^{\frac{\psi_n(p)}{2}}\int_{\partial\D}(h_n)_\tau e^{\frac{u_{\lambda_n}}{2}}y+\underbrace{e^{\frac{\psi_n(p)}{2}}\int_{\partial\D}(h_n)_\tau e^{\frac{u_{\lambda_n}}{2}}yO(\psi_n-\psi_n(p))}_{I}.$$
Using that $\|\psi_n\|_{C^{0,\alpha}}\leq C(1-\lambda_n)^{1-2\alpha}$, $\alpha\in (0,\frac{1}{2})$ we get
\begin{equation*}\begin{split}
|I|&\leq C(1-\lambda_n)^{1-2\alpha}\int_{\partial\D}e^{\frac{u_{\lambda_n}}{2}}|y||z-1|^\alpha\,dz\\
&\leq C(1-\lambda_n)^{2-2\alpha}\int_{\partial\D}\frac{|y||z-1|^{\alpha}}{\hat{\phi}_n^2(1-\lambda_n x)^2+\hat{\phi}_n^2(\lambda_n y)^2+\hat{k}_n(x-\lambda_n)^2+\hat{k}_ny^2}\,dz\\
&\leq C (1-\lambda_n)^{2-2\alpha},
\end{split}\end{equation*}
and $(i)$ follows. Likewise,
\begin{equation*}\begin{split}
\int_{\D}&\left((K_n)_y xy-(K_n)_x(1-x^2+y^2)\right)e^{u_n}\\
&\quad =e^{\psi_n(p)}\int_{\D}\left((K_n)_y xy-(K_n)_x(1-x^2+y^2)\right)e^{u_{\lambda_n}}\\
&\quad \;+\underbrace{e^{\psi_n(p)}\int_{\D}\left((K_n)_y xy-(K_n)_x(1-x^2+y^2)\right)e^{u_{\lambda_n}}O(\psi_n-\psi_n(p))}_{II},
\end{split}\end{equation*}
and
\begin{equation*}\begin{split}
|II|&\leq C(1-\lambda_n)^{2-2\alpha}\int_{\D}\frac{(|x||y|+|1-x|+|y|^2)|z-1|^{\alpha}(1+\lambda_n)}{(\hat{\phi}_n^2(1-\lambda_n x)^2+\hat{\phi}_n^2(\lambda_n y)^2+\hat{k}_n(x-\lambda_n)^2+\hat{k}_ny^2)^2}\,dz\\
&\leq C (1-\lambda_n)^{2-2\alpha},
\end{split}\end{equation*}
what proves $(ii)$.
\end{proof}

The next result finishes the proof of Theorem \ref{main}.

\begin{theorem}
Recall the definition of $\Phi$ given in \eqref{PHI}. Then
\begin{equation}\label{secondcond}
2(-\Delta)^{1/2}h(p)+\frac{K_x(p)}{\Phi(p)}=0,
\end{equation}
or, in other words, $\Phi_{\nu}(p)=0$.
\end{theorem}

\begin{proof}
Let $u_n(z)=u_{\lambda_n}(z)+\psi_n(z)$ be the solution of \eqref{ggn} that satisfies Theorem \ref{main}, $(ii)$, with $u_{\lambda_n}$ defined in \eqref{uLambda} for $\lambda_n\in (-1,1)$. By Proposition \ref{KW} there holds
 \begin{equation}\label{conduLambda0}
2\int_{\D}(K_n)_ye^{u_n}xy-\int_{\D}(K_n)_xe^{u_n}(1-x^2+y^2)=-4\int_{\partial\D}(h_n)_\tau e^{u_n/2}y,
\end{equation}
and, by Lemma \ref{asint},
\begin{equation}\begin{split}\label{conduLambda}
2\int_{\D}(K_n)_ye^{u_{\lambda_n}}xy-\int_{\D}(K_n)_x&e^{u_{\lambda_n}}(1-x^2+y^2)\\
&=-4\int_{\partial\D}(h_n)_\tau e^{u_{\lambda_n}/2}y+o_n(1)(1-\lambda_n).
\end{split}\end{equation}
Let us study first the term in the right hand side. Integrating by parts and substituting by the precise value of $u_{\lambda_n}$ we get
\begin{equation*}\begin{split}
-4\int_{\partial\D}&(h_n)_\tau e^{u_{\lambda_n}/2}y=-4\int_{\partial\D}(h_n-h_n(p))_\tau e^{u_{\lambda_n}/2}y\\
&=4\int_{\partial\D}(h_n-h_n(p))\left(e^{u_{\lambda_n}/2}\frac{(u_{\lambda_n})_\tau}{2}y+e^{u_{\lambda_n}/2}x\right)\\
&=8(1-\lambda_n^2)\frac{\hat{\phi}_n}{\hat{\phi}_n^2+\hat{k}_n}\int_{\partial\D}(h_n-h_n(p))\frac{x(1-\lambda_n)^2+2\lambda_n(x-1)}{((1-\lambda_n)^2-2\lambda_n(x-1))^2}.
\end{split}\end{equation*}
Consider now the left hand side in \eqref{conduLambda}. Again integrating by parts and using the definition of $u_{\lambda_n}$ it can be seen that
\begin{equation*}\begin{split}
2&\int_{\D}(K_n)_ye^{u_{\lambda_n}}xy-\int_{\D}(K_n)_xe^{u_{\lambda_n}}(1-x^2+y^2)\\
&=\int_{\D}(K_n-K_n(p))\left((e^{u_{\lambda_n}})_x(1-x^2+y^2)-2(e^{u_{\lambda_n}})_yxy-4xe^{u_{\lambda_n}}\right)\\
&=-16(1-\lambda_n^2)^2\hat{\phi}_n^2(\hat{\phi}_n^2+\hat{k}_n)\int_{\D}\frac{(K_n-K_n(p))(-\lambda_n((1-x)^2+y^2)+(1-\lambda_n)^2x)}{(\hat{\phi}_n^2(1-\lambda_n x)^2+\hat{\phi}_n^2(\lambda_n y)^2+\hat{k}_n(x-\lambda_n)^2+\hat{k}_ny^2)^3}.
\end{split}\end{equation*}
Calling
\begin{equation*}\begin{split}
I_n&:=\int_{\partial\D}(h_n-h_n(p))\frac{x(1-\lambda_n)^2+2\lambda_n(x-1)}{((1-\lambda_n)^2-2\lambda_n(x-1))^2},\\
II_n&:=(1-\lambda_n)\int_{\D}\frac{(K_n-K_n(p))(-\lambda_n((1-x)^2+y^2)+(1-\lambda_n)^2x)}{(\hat{\phi}_n^2(1-\lambda x)^2+\hat{\phi}_n^2(\lambda_n y)^2+\hat{k}_n(x-\lambda_n)^2+\hat{k}_ny^2)^3},
\end{split}\end{equation*}
the identity \eqref{conduLambda} is equivalent to
\begin{equation}\label{integralId}
I_n=-2(1+\lambda_n)\hat{\phi}_n(\hat{\phi}_n^2+\hat{k}_n)^2II_n+o_n(1).
\end{equation}
We aim to pass to the limit when $n\rightarrow\infty$ in this expression, that corresponds to the effect of concentration of the {\it bubble} at the point $p$. 

Let $\varepsilon >0$ fixed. We assume $1-\lambda_n <\varepsilon$, and we compute the limit of $I_n$ dividing the integral in three regions: $\partial\D\setminus B_\varepsilon (p)$, $\partial\D\cap (B_\varepsilon(p)\setminus B_{1-\lambda_n}(p))$ and $\partial\D\cap B_{1-\lambda_n}(p)$. Indeed, if we denote
$$f_n(x,y):= \frac{x(1-\lambda_n)^2+2\lambda_n(x-1)}{((1-\lambda_n)^2-2\lambda_n(x-1))^2},$$
we write
\begin{equation*}\begin{split}
I_n=&\int_{\partial\D\setminus B_\varepsilon (p)}(h_n-h_n(p))f_n+\int_{\partial\D\cap (B_\varepsilon(p)\setminus B_{1-\lambda_n}(p))}(h_n-h_n(p))f_n\\
&+\int_{\partial\D\cap B_{1-\lambda_n} (p)}(h_n-h_n(p))f_n\\
=:&\,I_1+I_2+I_3.
\end{split}\end{equation*}
Notice that in the region $\partial\D\setminus B_\varepsilon (p)$ the integral is no longer singular. Thus, for $\lambda_n$ sufficiently close to 1,
$$|(h_n-h_n(p))f_n|\leq C\|h\|_{L^\infty(\partial\D)}\in L^1(\partial\D\setminus B_\varepsilon (p)),$$
where $C$ is a positive constant independent of $n$. By the Dominated Convergence Theorem,
$$\lim_{n\rightarrow \infty}I_1=\int_{\partial\D\setminus B_\varepsilon (p)}\frac{h(p)-h}{2(1-x)}.$$
On the other hand, due to the evenness of $f_n$ with respect to the variable $y$, $I_2$ can be written as
$$I_2=\int_{\partial\D}(h_n-h_n(p)-(h_n)_\tau(p)y)f_n\chi_{\partial\D\cap (B_\varepsilon(p)\setminus B_{1-\lambda_n}(p))},$$
and using the facts that $h_n$ is uniformly bounded in $C^2(\partial\D)$ and $|x-1|=\frac{|z-1|^2}{2}>\frac{(1-\lambda_n)^2}{2}$ we get
$$|I_2|\leq C\int_{\partial\D}\frac{|y|^{2}|x-1|}{|x-1|^2}\chi_{B_\varepsilon (p)}\leq C\int_{\partial\D\cap B_\varepsilon(p)}\frac{|z-1|^2|z-1|^2}{|z-1|^4}\leq C \varepsilon.$$
Therefore,
$$\lim_{n\rightarrow \infty}I_2=O(\varepsilon).$$
Likewise,
\begin{equation*}\begin{split}
|I_3|&\leq C\int_{\partial\D\cap B_{1-\lambda_n}(p)}\frac{(1-\lambda_n)^2(1-\lambda_n)^2}{(1-\lambda_n)^4}\leq C(1-\lambda_n),
\end{split}\end{equation*}
where we have used again the evenness of $f_n$ with respect to $y$ and that in this region $|1-x|\leq C(1-\lambda_n)^2$ and $|y|^2\leq C(1-\lambda_n)^2$. Hence,
$$\lim_{n\rightarrow \infty}I_3=0.$$
We compute now the limit of $II_n$. Naming
$$g_n(x,y):=\frac{(-\lambda_n((1-x)^2+y^2)+(1-\lambda_n)^2x)}{\left(\hat{\phi}_n^2(1-\lambda_n x)^2+\hat{\phi}_n^2(\lambda_n y)^2+\hat{k}_n(x-\lambda_n)^2+\hat{k}_ny^2\right)^3},$$
we split the integral as
\begin{equation*}\begin{split}
II_n&=(1-\lambda_n)\int_{\D\setminus B_\rho(p)}(K_n-K_n(p))g_n+(1-\lambda_n)\int_{\D\cap B_\rho(p)}(K_n-K_n(p))g_n\\
&=II_1+II_2.
\end{split}\end{equation*}
We first consider the case of $\D\setminus B_\rho(p)$. Observe that for $0<\rho<1$ fixed and $z=(x,y)$ we have $|x-1|\geq c(\rho)>0$ for some constant dependent on $\rho$, and hence
$$|x-\lambda_n|=|x-1-\lambda_n+1|\geq|x-1|-|1-\lambda_n|\geq c(\rho)-|1-\lambda_n|\geq\frac{c(\rho)}{2}$$
for $\lambda_n$ sufficiently close to $1$. On the other hand,
$$(1-\lambda_n x)^2\geq (x-\lambda_n)^2,$$
and thus, provided that $\Phi(p)^2+K(p)>0$, for $\lambda_n$ close enough to $1$ we have that
$$|II_1|\leq C(1-\lambda_n)\|K\|_{L^\infty(\D)},$$
with $C$ a positive constant indepedent of $n$. Thus
$$\lim_{n\rightarrow \infty}II_1=0.$$
We finally estimate the term $II_2$. Using that
$$K_n(x,y)-K_n(p)-(K_n)_y(p)y=(K_n)_x(p)(x-1)+O(|z-1|^2),$$
and the evenness of $g_n$ with respect to $y$ we write the integral as
\begin{equation*}\begin{split}
II_2&=(1-\lambda_n)(K_n)_x(p)\int_{\D\cap B_\rho(p)}(x-1)g_n+(1-\lambda_n)\int_{\D\cap B_\rho(p)}O(|z-1|^2)g_n\\
&=:II_{21}+II_{22}.
\end{split}\end{equation*}
Making the change of variables
$$1-x=\tilde{x}(1-\lambda_n),\quad y=\tilde{y}(1-\lambda_n),$$
we obtain
$$II_{21}=(K_n)_x(p)\int_{(\R_+\times\R)\cap B_{\frac{\rho}{1-\lambda_n}}(0)}\frac{\tilde{x}(\tilde{x}(1-\lambda_n)-1+\lambda_n(\tilde{x}^2+\tilde{y}^2))\,d\tilde{y}\,d\tilde{x}}{(\hat{\phi}_n^2(1+\lambda_n \tilde{x})^2+\hat{\phi}_n^2(\lambda_n\tilde{y})^2+\hat{k}_n(1-\tilde{x})^2+\hat{k}_n\tilde{y}^2)^3}.$$
Assume first $K(p)> 0$. Notice that, since $\tilde{x}\geq 0$, $(1+\lambda_n\tilde{x})^2\geq 1+\lambda_n^2\tilde{x}^2$, and thus, for $\lambda_n$ close enough to 1,
\begin{equation*}\begin{split}
\bigg|\frac{\tilde{x}(\tilde{x}(1-\lambda_n)-1+\lambda_n(\tilde{x}^2+\tilde{y}^2))}{(\hat{\phi}_n^2(1+\lambda_n \tilde{x})^2+\hat{\phi}_n^2(\lambda_n\tilde{y})^2+\hat{k}_n(1-\tilde{x})^2+\hat{k}_n\tilde{y}^2)^3}&\chi_{(\R_+\times\R)\cap B_{\frac{\rho}{1-\lambda_n}}(0)}\bigg|\\
\leq C\frac{\tilde{x}(2(\tilde{x}^2+\tilde{y}^2)+1)}{(\Phi(p)^2+\Phi(p)^2\frac{\tilde{x}^2}{2}+K(p)\tilde{y}^2)^3}&,
\end{split}\end{equation*}
that belongs to $L^1(\R_+\times \R)$. Otherwise, if $K(p)\leq 0$ there holds (for $\lambda_n$ close enough to $1$)
\begin{equation*}\begin{split}
\hat{\phi}_n^2(1+\lambda_n\tilde{x})^2+\hat{k}_n(1-\tilde{x})^2&=\hat{\phi}_n^2+\hat{k}_n+ \tilde{x}^2(\hat{\phi}_n^2\lambda_n^2+\hat{k}_n)-2\tilde{x}(-\lambda_n\hat{\phi}^2+\hat{k}_n)\\
&\geq \hat{\phi}_n^2+\hat{k}_n + \tilde{x}^2(\hat{\phi}_n^2\lambda_n^2+\hat{k}_n),
\end{split}\end{equation*}
since $\tilde{x}\geq 0$ and $\Phi(p)^2+K(p)>0$. Thus, there exists $c_0>0$, independent of $\lambda_n$, such that 
\begin{equation*}\begin{split}
\bigg|\frac{\tilde{x}(\tilde{x}(1-\lambda_n)-1+\lambda_n(\tilde{x}^2+\tilde{y}^2))}{(\hat{\phi}_n^2(1+\lambda_n \tilde{x})^2+\hat{\phi}_n^2(\lambda_n\tilde{y})^2+\hat{k}_n(1-\tilde{x})^2+\hat{k}_n\tilde{y}^2)^3}&\chi_{(\R_+\times\R)\cap B_{\frac{\rho}{1-\lambda_n}}(0)}\bigg|\\
\leq c_0\frac{\tilde{x}(2(\tilde{x}^2+\tilde{y}^2)+1)}{(1+\tilde{x}^2+\tilde{y}^2)^3}&\in L^1(\R_+\times \R).
\end{split}\end{equation*}
Passing to the limit in the integral we conclude
\begin{equation*}\begin{split}
\lim_{n\rightarrow \infty}II_{21}&=K_x(p)\left[\int_{\R_+\times\R}\frac{\tilde{x}((\tilde{x}^2+\tilde{y}^2)-1)\,d\tilde{y}\,d\tilde{x}}{(\Phi(p)^2(1+\tilde{x})^2+(\Phi(p)^2+K(p))\tilde{y}^2+K(p)(1-\tilde{x})^2)^3}\right]\\
&=\frac{\pi}{8\Phi(p)^2(K(p)+\Phi(p)^2)^2}K_x(p).
\end{split}\end{equation*}
The above integral in the half-plane has been computed with the help of Mathematica. 

Moreover,
\begin{equation*}\begin{split}
|II_{22}|&\leq C(1-\lambda_n)\int_{(\R_+\times\R)\cap B_{\frac{\rho}{1-\lambda_n}}(0)}\frac{(\tilde{x}^2+\tilde{y}^2)|\tilde{x}(1-\lambda_n)-1+\lambda_n(\tilde{x}^2+\tilde{y}^2)|\,d\tilde{y}\,d\tilde{x}}{|\hat{\phi}_n^2(1+\lambda_n \tilde{x})^2+\hat{\phi}_n^2(\lambda_n\tilde{y})^2+\hat{k}_n(1-\tilde{x})^2+\hat{k}_n\tilde{y}^2|^3}\\
&\leq C(1-\lambda_n)\int_{(\R_+\times\R)\cap B_{\frac{\rho}{1-\lambda_n}}(0)}\frac{(\tilde{x}^2+\tilde{y}^2)(2(\tilde{x}^2+\tilde{y}^2)+1)\,d\tilde{y}\,d\tilde{x}}{(1+\tilde{x}^2+\tilde{y}^2)^3}\\
&=O\left((1-\lambda_n)\log(1-\lambda_n)\right),
\end{split}\end{equation*}
and therefore
$$\lim_{n\rightarrow \infty}II_{22}=0.$$
Passing to the limit in \eqref{integralId} we get
$$\int_{\partial\D\setminus B_\varepsilon (p)}\frac{h(p)-h}{2(1-x)}+O(\varepsilon)=-\frac{\pi}{2\Phi(p)}K_x(p).$$
Finally, making $\varepsilon\rightarrow 0$,
$$\mbox{p.v.}\int_{\partial\D}\frac{h(p)-h}{2(1-x)}=-\frac{\pi}{2\Phi(p)}K_x(p).$$
Taking into account the definition of $(-\Delta)^{1/2}$ in $\mathbb{S}^1$ (see, for instance, \cite[Appendix]{DLMR}), we obtain:
$$2(-\Delta)^{1/2}h(p)+\frac{K_x(p)}{\Phi(p)}=0.$$
\end{proof}


\begin{remark} \label{interpretation} As we have seen, the point of concentration of a blow-up sequence is a critical point of the function $\Phi$. This phenomenon has a nice interpretation. Consider the energy functional associated to \eqref{gg}, that is, $I: H^1(\D) \to \R$,
$$ I(u):=  \int_{\D} \frac 1 2 |\nabla u|^2 - 2 K e^u + \int_{\partial \D} 2 u - 4 h e^{u/2}.$$
Let us evaluate the energy of the functions $u_{\lambda_n}$ defined in \eqref{uLambda}, as $\lambda_n \to 1$. Because of the concentration phenomena, we have
$$I(u_{\lambda_n})= \tilde{I}(u_{\lambda_n}) + o_n(1),$$
where $\tilde{I}$ is the limit functional related to constant curvatures
$$ \tilde{I}(u):=  \int_{\D} \frac 1 2 |\nabla u|^2 - 2 K(p) e^u + \int_{\partial \D} 2 u - 4 h(p) e^{u/2}.$$
Observe now that $\tilde{I}(u_{\lambda_n})$ is constant in $\lambda_n$; so, for convenience, we may take $\lambda_n=0$, that is,
$$u_0(z)= 2 \log \left (\frac{2\Phi(p)}{\Phi(p)^2+K(p)|z|^2} \right ).$$
Let us compute the terms
\begin{equation*}\begin{split} \int_{\D} \frac 1 2 |\nabla u_0|^2 &= 8 \frac{K(p)^2}{\Phi(p)^4} \, 2 \pi \int_0^1 \frac{r^3 \, dr}{(1+\frac{K(p)}{\Phi(p)^2}r^2)^2} = 8 \pi \int_0^{\frac{K(p)}{\Phi(p)^2}} \frac{s \, ds}{(1+s)^2} \\
&= 8 \pi \left( \log\left(1+\frac{K(p)}{\Phi(p)^2}\right) + \frac{\Phi(p)^2}{\Phi(p)^2+K(p)}-1\right)\end{split}\end{equation*}
$$ \int_{\partial \D} 2 u_0= 8 \pi \log \frac{2\Phi(p)}{\Phi(p)^2+K(p)}.$$
Moreover, we already know the value of the exponential terms, recall \eqref{conver1}. Then,
\begin{equation*}\begin{split}
\tilde{I}(u_0) &= 8 \pi \left ( \log(2 /\Phi(p)) + \frac{\Phi(p)^2}{\Phi(p)^2+K(p)}-1\right ) - 4 \pi \left (1+ \frac{h(p)}{h^2(p)+K(p)} \right ) \\
&= -8 \pi + 4 \pi \left ( 2 \log(2 /\Phi(p)) + \frac{2\Phi(p)^2}{\Phi(p)^2+K(p)}-1 - \frac{h(p)}{h^2(p)+K(p)} \right ).
\end{split}\end{equation*}
By definition $\frac{2\Phi(p)^2}{\Phi(p)^2+K(p)}-1 - \frac{h(p)}{h^2(p)+K(p)} =0$ and hence
\begin{equation*}\begin{split}
\tilde{I}(u_0)&=  -8 \pi (1 + \log (\Phi(p)/2)).
\end{split}\end{equation*}
\end{remark}

\section{Appendix}\label{Appendix}
\setcounter{equation}{0}

\begin{proposition}\label{appendix}
Let $1<q<2$ and $a_n\in\D$ with $|a_n|\to 1$ as $n\to+\infty$. There exists $C>0$, independent of $n$, such that
\begin{itemize}
\item[i)]$\displaystyle \int_{\D}\left|f_{a_n}(z)-f_{a_n}\left(\frac{a_n}{|a_n|}\right)\right|^q\leq C(1-|a_n|)^{q}$,
\item[ii)]$\displaystyle \int_{\partial\D}\left|f_{a_n}(z)-f_{a_n}\left(\frac{a_n}{|a_n|}\right)\right|^q\leq C(1-|a_n|),$
\end{itemize}
where $f_a(z):=\frac{a+z}{1+\bar{a}z}$.
\end{proposition}

\begin{proof}
By definition of $f_a(z)$,
\begin{equation*}\begin{split}
\bigg|f_{a_n}(z)-f_{a_n}\left(\frac{a_n}{|a_n|}\right)\bigg|&=\bigg|\frac{a_n+z}{1+\overline{a_n}z}-\frac{a_n}{|a_n|}\bigg|=\frac{(1-|a_n|)}{|a_n|}\bigg|\frac{a_n-|a_n|z}{1+\overline{a_n}z}\bigg|\\
&\leq C(1-|a_n|)\frac{1}{|1+\overline{a_n}z|},
\end{split}\end{equation*}
for $|a_n|$ sufficiently close to 1.
Thus,
\begin{equation}\label{boundLp}
\int_{\D}\left|f_{a_n}(z)-f_{a_n}\left(\frac{a_n}{|a_n|}\right)\right|^q\leq C(1-|a_n|)^q\int_{\D}\frac{1}{|1+\overline{a_n}z|^q},
\end{equation}
that is singular at $z_0=-\frac{1}{\overline{a_n}}=-\frac{a_n}{|a_n|^2}$. Notice that $|z_0|>1$ and therefore $z_0\notin\D$. Since $1<q<2$ we inmediately obtain $i)$. \medskip

Let us estimate the integral on $\partial\D$. Notice first that, if $z\in\partial\D$, then
$$|1+\overline{a_n}z|=|\overline{z}||1+\overline{a_n}z|=|\overline{z}+\overline{a_n}|=|z+a_n|=\mbox{dist}(z,-a_n).$$
Thus, arguing as before we obtain 
\begin{equation}\label{boundLpS}\begin{split}
\int_{\partial\D} |f_{a_n}(z)-&f_{a_n}\left(\frac{a_n}{|a_n|}\right)|^q\leq C(1-|a_n|)^q\int_{\partial\D}\frac{1}{|1+\overline{a_n}z|^q}\\
&=C(1-|a_n|)^q\int_{\partial\D}\frac{1}{\mbox{dist}(z,-a_n)^q}.
\end{split}\end{equation}
Notice that
$$\mbox{dist}(z,-a_n)\geq \mbox{dist}\left(-\frac{a_n}{|a_n|},-a_n\right)=1-|a_n|.$$
We split the integral in two regions,
$$\S_1:=\partial \D\setminus B_1\left(-\frac{a_n}{|a_n|}\right),\qquad \S_2:=\partial \D\cap B_1\left(-\frac{a_n}{|a_n|}\right).$$
If $z\in\S_1$ then there exists $c>0$, independent of $n$, such that dist$(z,-a_n)\geq c$. Therefore
\begin{equation}\label{intS1}
\int_{\S_1}\frac{1}{\mbox{dist}(z,-a_n)^q}\leq c^{-q}|\S_1|\leq c^{-q}|\partial\D|.
\end{equation}
To analyze the region $\S_2$ we divide the integral in subintervals in the following form,
$$L_1:=B_{(1-|a_n|)}\left(-\frac{a_n}{|a_n|}\right)\cap\S_2,$$
$$L_j:=\left(B_{j(1-|a_n|)}\left(-\frac{a_n}{|a_n|}\right)\setminus B_{(j-1)(1-|a_n|)}\left(-\frac{a_n}{|a_n|}\right)\right)\cap\S_2,\quad j=2,\ldots, J_n,$$
where $J_n:=\lceil\frac{1}{1-|a_n|}\rceil$, and
$$\{z_j,\tilde{z}_j\}:=\partial B_{j(1-|a_n|)}\left(-\frac{a_n}{|a_n|}\right)\cap\S_2,\quad j=1,\ldots,J_n-1.$$

\begin{figure}[h!]
     \centering
           \begin{subfigure}[b]{0.3\textwidth}
                \includegraphics[width=\textwidth]{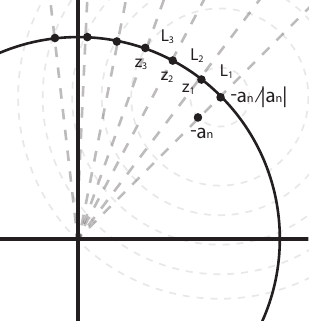}
                \caption*{Domain $\Sigma_2$.}
       \end{subfigure}
\end{figure}
\noindent Notice that
$$|L_j|=2\arcsin(1-|a_n|)\approx 2(1-|a_n|),$$
and
$$\mbox{dist}(z,-a_n)\geq 1-|a_n|,\;\;z\in L_1,$$
$$\mbox{dist}(z,-a_n)\geq \mbox{dist}(z_{j-1},-a_n)\geq (j-1)(1-|a_n|),\;\; z\in L_j,$$
for all $j=2,\ldots,J_n$. Thus,
$$\int_{L_1}\frac{1}{\mbox{dist}(z,-a_n)^q}\leq \frac{2|L_1|}{(1-|a_n|)^q}\leq C(1-|a_n|)^{1-q},$$
$$\int_{L_j}\frac{1}{\mbox{dist}(z,-a_n)^q}\leq \frac{C}{(j-1)^q}(1-|a_n|)^{1-q},\;\;j=2,\ldots,J_n.$$
Summing up we get
\begin{equation}\begin{split}\label{intS2}
\int_{\S_2}&\frac{1}{\mbox{dist}(z,-a_n)^q}\leq C(1-|a_n|)^{1-q}\left(1+\sum_{j=2}^{J_n}\frac{1}{(j-1)^q}\right)\\
&\;\leq C(1-|a_n|)^{1-q}\left(1+\sum_{j=1}^{+\infty}\frac{1}{j^q}\right)\leq C(1-|a_n|)^{1-q}.
\end{split}\end{equation}
Replacing \eqref{intS1} and \eqref{intS2} in \eqref{boundLpS} we conclude the validity of $ii)$.
\end{proof}

\bibliographystyle{unsrt}

\end{document}